\documentclass{gtmon_a}
\pdfoutput=1
\usepackage{graphicx,longtable,pinlabel,rotating}

\proceedingstitle{Heegaard splittings of 3--manifolds (Haifa 2005)}
\conferencestart{10 July 2005}
\conferenceend{19 July 2005}
\conferencename{Heegaard splittings of 3--manifolds}
\conferencelocation{Haifa}

\editor{Cameron Gordon}
\givenname{Cameron}
\surname{Gordon}

\editor{Yoav}
\givenname{Yoav}
\surname{Moriah}

\title[Kleinian orbifolds uniformized by $\mathcal{RP}$ groups]
{Kleinian orbifolds uniformized by $\mathcal{RP}$ groups with
an elliptic and a hyperbolic generator}

\author[E Klimenko]{Elena Klimenko}
\givenname{Elena}
\surname{Klimenko}
\address{Max Planck Institute for Mathematics\\\newline
Vivatsgasse 7\\53111 Bonn\\Germany}
\email{klimenko@mpim-bonn.mpg.de}
\urladdr{}

\author[N Kopteva]{Natalia Kopteva}
\givenname{Natalia}
\surname{Kopteva}
\address{Sobolev Institute of Mathematics \\\newline
Acad. Koptyug ave., 4\\Novosibirsk 630090\\Russia}
\email{natasha@math.nsc.ru}
\urladdr{}

\volumenumber{12}
\issuenumber{}
\publicationyear{2007}
\papernumber{05}
\startpage{121}
\endpage{156}

\doi{}
\MR{}
\Zbl{}

\arxivreference{}  


\received{6 October 2006}
\revised{16 March 2007}
\accepted{30 April 2007}
\published{3 December 2007}
\publishedonline{3 December 2007}
\proposed{}
\seconded{}
\corresponding{}
\version{}

\makeatletter
\def\cnewtheorem#1[#2]#3{\newtheorem{#1}{#3}[section]
\expandafter\let\csname c@#1\endcsname\c@theorem}

\let\xysavmatrix\xymatrix
\def\xymatrix{\disablesubscriptcorrection\xysavmatrix}
\AtBeginDocument{\let\bar\wbar\let\tilde\wtilde\let\hat\what\let\widetilde\wwtilde}

\makeop{Tet}
\makeop{orb}

\newtheorem{theorem}{Theorem}[section]
\cnewtheorem{lemma}[theorem]{Lemma}
\cnewtheorem{proposition}[theorem]{Proposition}
\cnewtheorem{cor}[theorem]{Corollary}
\theoremstyle{remark}
\cnewtheorem{rem}[theorem]{Remark}
\makeatother  

\renewcommand{\theequation}{\thesection.\arabic{equation}}

\def\tr{\mathrm{tr}\,}
\def\PSL{\mathrm{PSL}(2,{\mathbb C})}

\def\P{\mathcal{P}}
\def\Q{\mathcal{Q}}

\begin{document}

\begin{asciiabstract}
We consider non-elementary Kleinian groups \Gamma without invariant plane
generated by an elliptic and a hyperbolic elements with their axes lying 
in one plane.  We find presentations and a complete list of orbifolds
uniformized by such \Gamma.
\end{asciiabstract}

\begin{htmlabstract}
We consider non-elementary Kleinian groups &Gamma; without invariant plane
generated by an elliptic and a hyperbolic elements with their axes lying
in one plane.  We find presentations and a complete list of orbifolds
uniformized by such &Gamma;.
\end{htmlabstract}

\begin{abstract}    
We consider non-elementary Kleinian groups $\Gamma$, without invariant plane,
generated by an elliptic and a hyperbolic element, with their axes lying 
in one plane.  We find presentations and a complete list of orbifolds
uniformized by such $\Gamma$.
\end{abstract}

\maketitle

This work is a part of a program to describe all
2--generator
Kleinian groups with real parameters. 
We study $\mathcal{RP}$
groups,
that is, marked 2--generator subgroups 
$\Gamma=\langle f,g\rangle$
of  $\PSL$  for which the generating pair $(f,g)$ has real 
parameters 
$\beta=\beta(f)=\tr^2f-4$, $\beta'=\beta(g)=\tr^2g-4$
and $\gamma=\gamma(f,g)=\tr[f,g]-2$ (see \fullref{sec:prelim}
 for exact definitions).
Since discreteness questions were answered for elementary groups
and for
groups with invariant hyperbolic plane (in particular, all
Fuchsian
groups were described), we concentrate only on the non-elementary
$\mathcal{RP}$ groups without invariant plane, which we call
{\em truly spatial\/} $\mathcal{RP}$ groups.

This paper deals with the most complicated case of $\mathcal{RP}$ groups, the case
with one generator elliptic and
the other one hyperbolic.
It was shown by Klimenko and Kopteva \cite{KK02} that
`truly spatial' for this class means that
the elliptic generator is not a half-turn and the axes
of the generators either (1) are disjoint (non-parallel) 
lines lying in a
hyperbolic plane
or (2) intersect non-orthogonally at a point of ${\mathbb H}^3$.
In terms of parameters, we have here $\beta \in (-4,0)$,
$\beta' \in (0,\infty)$ and $\gamma$ for (1) and (2) belongs to
the intervals
 $(-\infty, 0)$ and $(0,-\beta\beta'/4)$, respectively
\cite[Theorem 1 and Table 1]{KK02}.
In the previous papers by Klimenko \cite{Kli90}
and Klimenko and Kopteva \cite{KK02,KK04}
necessary and sufficient conditions for
discreteness of all such groups were found constructively. Here
we use
the construction (we reproduce it in \fullref{sec:fp})
to determine fundamental polyhedra, presentations
and orbifolds for all truly spatial discrete $\mathcal{RP}$ groups
with an elliptic and a hyperbolic generators
(\fullref{sec:orbifolds}).

The other cases of $f$ and $g$ with real traces that generate a
truly spatial
$\mathcal{RP}$ group and the question when the group is discrete
were
investigated earlier by Klimenko \cite{Kli89,Kli01}
and Klimenko and Kopteva \cite{KK04-1}.
The final results including the results of the present paper are
collected in 
Klimenko and Kopteva \cite{KK05-2} (mostly without proofs),
where parameters, presentations and orbifolds for
{\em all truly spatial\/} discrete $\mathcal{RP}$ groups 
{\em with real traces of the generators\/} are given.

\subsection*{Acknowledgements}
N Kopteva would like to thank Gettysburg College
for sincere hospitality during her stay in October 2004 when
an essential part of the work was done. Both authors are grateful
to Prof F Grunewald and 
an anonymous referee for various comments on an earlier version
of this paper.
E Klimenko was partially supported by Gettysburg
College and N Kopteva was
partially supported by FP6 Marie Curie IIF Fellowship.
We also thank the DFG Forschergruppe ``Classification of surfaces",
Heinrich-Heine-Universit\"at D\"usseldorf and the Max-Planck-Institut
in Bonn for financial support.

\section{Preliminaries}\label{sec:prelim}
\subsubsection*{Definitions and notation}

We identify $\PSL$ with the full group of orientation preserving
isometries of
hyperbolic 3--space ${\mathbb H}^3$.

Let  $f,g\in\PSL$. The complex numbers 
$\beta=\beta(f)=\mathrm{tr}^2f-4$, $\beta'=\beta(g)=\mathrm{tr}^2g-4$
and
$\gamma=\gamma(f,g)=\mathrm{tr}[f,g]-2$,
where $[f,g]$ denotes the commutator $fgf^{-1}g^{-1}$, 
are called the parameters for the pair $(f,g)$ and for the group
$\Gamma=\langle f,g\rangle$.

The same 2--generator subgroup $\Gamma$ of $\PSL$ can have different
triples of parameters $(\beta,\beta',\gamma)$ depending on the
choice of the generating pair $(f,g)$.
On the other hand, the triple of parameters 
$(\beta,\beta',\gamma)$ determines $\Gamma$ up to conjugacy
whenever $\gamma\not=0$.
More precisely, if $(f_1,g_1)$ and $(f_2,g_2)$ both have the same
triple of parameters $(\beta,\beta',\gamma)$ with $\gamma\not=0$,
then there is $h\in\PSL$ so that $f_2=hf_1h^{-1}$ and either
$g_2=hg_1h^{-1}$ or $g_2=hg_1^{-1}h^{-1}$, see 
Gehring and Martin \cite{GM89}.

Notice that if $\gamma=0$ then $\Gamma$ is not determined
uniquely by the numbers $\beta$ and $\beta'$.
There are examples of a discrete group $\Gamma_1$
and a non-discrete $\Gamma_2$ with $\gamma=0$ and the same pair
$(\beta,\beta')$.
However, it is known that in this case $f$ and $g$
have a common fixed point in $\partial {\mathbb H}^3$, 
that is, $\Gamma$
is elementary. Since we are concerned only with truly spatial groups,
we may assume that $\gamma\not=0$ throughout this paper.

A triple $(\Gamma;f,g)$, where $\Gamma=\langle f,g\rangle$,
is called an $\mathcal{RP}$ {\em group\/} if the pair $(f,g)$ has
real parameters $(\beta,\beta',\gamma)$. 
Note that the requirement of discreteness is not included in the
definition
of an $\mathcal{RP}$ group.

We recall that a non-trivial 
element $f\in\PSL$ with real $\beta=\beta(f)$
is
{\em elliptic\/}, {\em parabolic\/}, {\em hyperbolic\/} or
{\em $\pi$--loxodromic\/}
according to whether
$\beta\in[-4,0)$, $\beta=0$, $\beta\in(0,+\infty)$ or
$\beta\in(-\infty,-4)$.
If $\beta\notin[-4,\infty)$, ie if $\mathrm{tr} f$ is not real,
then $f$ is called {\em strictly loxodromic\/}.

An elliptic element $f$ of order $n$ is said to be {\em
primitive\/}
if it is a rotation through $2\pi/n$ (with
$\beta=-4\sin^2(\pi/n)$); otherwise, it is
called {\em non-primitive\/} (and then $\beta=-4\sin^2(\pi q/n)$,
where $q$ and $n$
are coprime and $1<q<n/2$).

A hyperbolic plane divides ${\mathbb H}^3$ into two components;
we shall call the closure in ${\mathbb H}^3$
of either of them a {\em half-space\/}
in ${\mathbb H}^3$.
A connected subset $P$ of ${\mathbb H}^3$ with non-empty interior
is said to be
a {\em (convex) polyhedron\/} if it is the intersection of a family
$\cal H$
of half-spaces with the property that each point of $P$
has a neighborhood meeting at most a finite number of
boundaries of elements of $\cal H$.
A closed polyhedron with finite number of faces
bounded by planes $\alpha_1,\dots,\alpha_k$ is denoted
by $\P(\alpha_1,\dots,\alpha_k)$.

We define a {\em tetrahedron\/} $T$ to be a polyhedron which in the
projective
ball model is the intersection of the hyperbolic space ${\mathbb
H}^3$
with a Euclidean tetrahedron $T_E$ (possibly with vertices on
the sphere $\partial {\mathbb H}^3$ at infinity or beyond it) so
that the
intersection of each edge of $T_E$ with ${\mathbb H}^3$ is
non-empty.

A tetrahedron $T$ (possibly of infinite volume) in ${\mathbb H}^3$
is uniquely determined up to isometry by its dihedral angles.
Let $T$  have dihedral angles $\pi/p_1$, $\pi/p_2$, $\pi/p_3$
at the edges of some face and let $\pi/q_1$, $\pi/q_2$, $\pi/q_3$ be
dihedral angles of $T$ that are opposite to
$\pi/p_1$, $\pi/p_2$, $\pi/p_3$, respectively.
Then a standard notation for such a $T$ is
$T[p_1,p_2,p_3;q_1,q_2,q_3]$ and a standard notation for the group
generated by reflections in the faces of $T$ is $G_T$.

We denote the reflection in a plane $\kappa$ by $R_\kappa$.
The axis of an element $h\in\PSL$ with two distinct fixed points
in $\partial{\mathbb H}^3$ is denoted by the same $h$ if this does
not lead
to any confusion.

We use symbols $\infty$ and $\overline\infty$ with the following
convention.
We assume that $\overline\infty>\infty>x$ and
$x/\infty=x/\overline\infty=0$
for every real $x$;
$\infty/x=\infty$ and $\overline\infty/x=\overline\infty$
for every positive real $x$;
in particular, $(\infty,k)=(\overline\infty,k)=k$ for every
positive integer $k$. 
We use $(\cdot,\cdot)$ for $\mathrm{gcd}(\cdot,\cdot)$.

If we denote the dihedral angle between two planes by $\pi/p$
($1< p\leq\overline\infty$),
then the planes intersect when $p$ is finite, they are parallel
(that is, their closures in 
$\overline{{\mathbb H}^3}=\partial{\mathbb H}^3\cup{\mathbb H}^3$
have just one common point in $\partial{\mathbb H}^3$)
when $p=\infty$ and disjoint
(that is, the boundaries of the planes 
do not intersect in $\partial{\mathbb H}^3$)
when $p=\overline\infty$.

\subsubsection*{Convention on pictures}

Since the methods we use here are essentially geometrical, the
paper contains many pictures of hyperbolic polyhedra.
In those pictures, shaded polygons are not faces of polyhedra,
but are drawn to underline the combinatorial structure of the
corresponding polyhedron. They are just intersections of
the polyhedron with appropriate planes.

If a line on a picture is an edge of a polyhedron, then it is
labelled by the dihedral angle at this edge. We often omit labels
 $\pi/2$.
If a line is not an edge of a polyhedron and is labelled by an
integer $k$, then this means that the line is the axis of an elliptic
element of order $k$ that belongs to $\Gamma^*$ (see below).
\fullref{covering} is an exception from this convention. We shall
explain labels in \fullref{covering} in \fullref{rem_cover}.

\section{Fundamental polyhedra and parameters}\label{sec:fp}

From here on $f$ is a primitive elliptic element and $g$ is
hyperbolic.
The main tool in the study of discreteness of $\Gamma=\langle
f,g\rangle$ in Klimenko \cite{Kli90} and
Klimenko--Kopteva \cite{KK02,KK04} 
was a construction of a `convenient'
finite index extension
$\Gamma^*$ of $\Gamma$ together with a fundamental polyhedron
for
each discrete $\Gamma^*$.
In this section, we reproduce the construction of
$\Gamma^*$  and describe the fundamental polyhedra for
all discrete $\Gamma^*$. This is a preliminary part for
\fullref{sec:orbifolds},
where we shall work with the groups $\Gamma^*$ themselves to list
the corresponding
orbifolds.

\subsection{Geometric description of discrete groups for the case
of disjoint axes}\label{construction_disjoint}

\fullref{discr_disjoint} below gives necessary and sufficient
conditions for
discreteness of $\Gamma$ for the case of {\em disjoint axes\/} of
the generators
$f$ and $g$; a complete proof can be found in Klimenko \cite{Kli90}.
We also repeat the geometric construction from \cite{Kli90}
and recall fundamental polyhedra for the series of discrete
groups
$\Gamma^*$ corresponding to Items (2)(i)--(2)(iii) of
\fullref{discr_disjoint}.

\begin{theorem}[{\rm \cite{Kli90}}]\label{discr_disjoint}
Let $f\in\PSL$ be a primitive elliptic element of order $n\geq
3$,
$g\in\PSL$ be a hyperbolic element and let their axes be
disjoint
lines lying in a hyperbolic plane.
\begin{enumerate}
\item There exists $h\in\PSL$ such that
$h^2=fgf^{-1}g^{-1}$
and $(hg)^2=1$.
\item $\Gamma=\langle f,g\rangle$ is discrete if and only
if
one of the following holds:
\begin{itemize}
\item[{\rm(i)}]
$h$ is a hyperbolic, parabolic or primitive elliptic element
of order $p\geq 3$;
\item[{\rm (ii)}]
$n\geq 5$ is odd, $h=x^2$, where $x$ is a primitive elliptic
element of order $n$,
and $y=hgfx^{-1}f$ is a  hyperbolic, parabolic
or primitive elliptic element of order $q\geq 4$ or
\item[{\rm (iii)}]
$n=3$, $h=x^2$, where $x$ is a primitive elliptic element of
order $5$,
and $z=hgf(x^{-1}f)^3$ is a hyperbolic, parabolic or
primitive elliptic element of order $r\geq 3$.
\end{itemize}
\end{enumerate}
\end{theorem}

Let $f$ and $g$ be as in \fullref{discr_disjoint}, and let
$\omega$ be the plane in which
the (disjoint) axes of $f$ and $g$ lie.

Denote by $\varepsilon$ the plane that passes
through the common perpendicular
to the axes of $f$ and $g$ orthogonally to $\omega$.
Let $\alpha$ and $\tau$ be the planes such that
$f=R_\alpha R_\omega$ and $g=R_\tau R_\varepsilon$,
and let $\P=\P(\omega,\varepsilon,\alpha,\tau)$.
The planes $\omega$ and $\alpha$ make a dihedral angle
of $\pi/n$;
the planes $\varepsilon$ and $\tau$ are disjoint so that the axis
of $g$
is their common perpendicular. Moreover, $\alpha$ is orthogonal
to $\varepsilon$ and $\tau$ is orthogonal to $\omega$.
The planes $\alpha$ and $\tau$ either intersect non-orthogonally
or are parallel or disjoint. We denote the dihedral angle of
$\P$ between these planes by
$\pi/p$, $p>2$, where, by convention, 
$p=\infty$ if $\alpha$ and $\tau$
are parallel and $p=\overline\infty$ if they are disjoint.

We consider two finite index extensions of 
$\Gamma=\langle f,g\rangle$:
$\widetilde\Gamma=\langle f,g,e\rangle$, where $e=R_\varepsilon
R_\omega$, and
$\Gamma^*=\langle f,g,e,R_\omega\rangle$.
$\widetilde\Gamma$ is the orientation preserving subgroup of
index $2$ in $\Gamma^*$
and $\widetilde\Gamma$ contains $\Gamma$ as a subgroup of index
at most $2$. In \fullref{sec:orbifolds}, we shall see when
$\Gamma=\widetilde\Gamma$
and when $\Gamma\not=\widetilde\Gamma$.

It was shown in \cite{Kli90} that $h=R_\alpha R_\tau$
is the only element that satisfies both $h^2=[f,g]$ and
$(hg)^2=1$.
There are three series of discrete groups $\Gamma^*$
depending on how $\P$ is decomposed into fundamental polyhedra
for $\Gamma^*$.
The series correspond to the conditions
(2)(i), (2)(ii) and (2)(iii) of \fullref{discr_disjoint}.

{\bf 1.} \qua
$h$ is a hyperbolic, parabolic or primitive elliptic element
of order~$p\geq 3$ (that is (2)(i) holds)
if and only if the dihedral angle of $\P$ between $\alpha$
and $\tau$ is of the form $\pi/p$ with
$p=\overline\infty$, $p=\infty$, or
$p\in{\mathbb Z}$, $p\geq 3$, respectively.
This is the first series of the discrete groups.
In this case the polyhedron $\P$ is
a fundamental polyhedron for $\Gamma^*$.
In \fullref{fp_disjoint}(a) $\P$ is drawn under assumption
that
$1/n+1/p>1/2$.

\begin{figure}[ht!]
\centering
\begin{tabular}{cc}
\labellist
\small
\pinlabel {\turnbox{45}{$\pi/p$}} at 50 105
\pinlabel {$\tau$} at 40 70
\pinlabel {$\omega$} at 35 46
\pinlabel {\turnbox{-45}{$\pi/n$}} at 0 40
\pinlabel {$\alpha$} at 62 22
\pinlabel {$f$} at 35 15
\pinlabel {$\varepsilon$} at 100 38
\pinlabel {$g$} at 115 65 
\endlabellist
\includegraphics[width=4.6 cm]{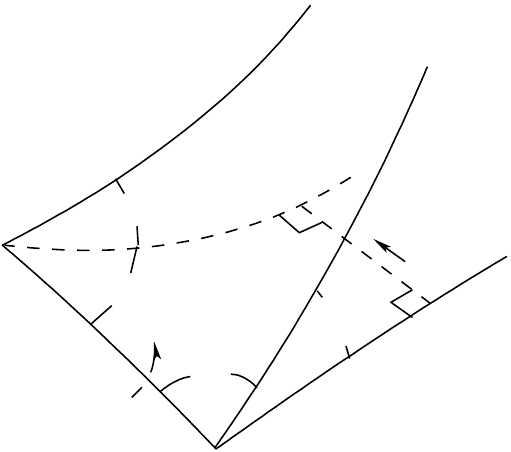} \qquad &\qquad
\labellist
\small
\pinlabel{\turnbox{30}{$2\pi/n$}} at 55 115
\pinlabel {$B$} at 5 75
\pinlabel {\turnbox{-45}{$\pi/n$}} at 40 40
\pinlabel {$\tau$} at 65 100
\pinlabel {$\omega$} at 65 60
\pinlabel {$3$} at 90 60 
\pinlabel {$A$} at 85 5
\pinlabel {$\xi_1$} at 110 75
\pinlabel {$\alpha$} at 92 30 
\pinlabel {$E$} at 95 100
\pinlabel {$\varepsilon$} at 128 47
\pinlabel {$C$} at 155 95 
\pinlabel {$q$} at 163 80
\pinlabel {$D$} at 170 55
\endlabellist
\includegraphics[width=5.6 cm]{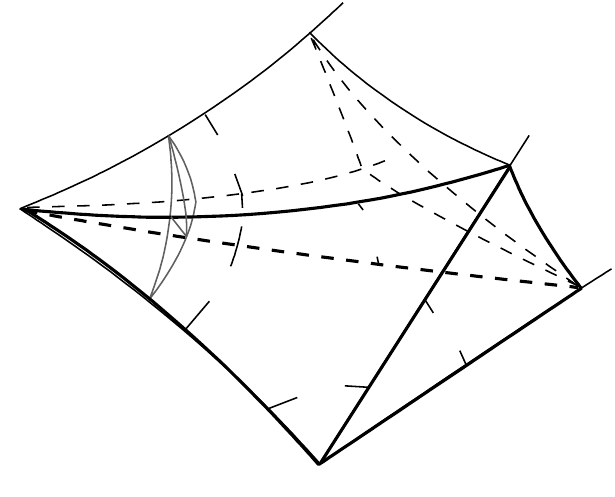}\\
(a) \qquad &\qquad (b)\\
\multicolumn{2}{c}{
\labellist
\small
\pinlabel {\turnbox{60}{$2\pi/5$}} at 15 80
\pinlabel {\turnbox{-60}{$\pi/3$}} at 10 23
\pinlabel {$\alpha$} at 33 17
\pinlabel {$\xi_2$} at 40 33
\pinlabel {$3$} at 38 68 
\pinlabel {$2$} at 50 43
\pinlabel {$2$} at 70 40 
\pinlabel {\turnbox{-30}{$r=3$}} at 53 29
\pinlabel {$5$} at 60 15
\pinlabel {$5$} at 90 45
\pinlabel {$\varepsilon$} at 48 13
\pinlabel {$\omega$} at 107 38
\endlabellist
\includegraphics[width=5.6cm]{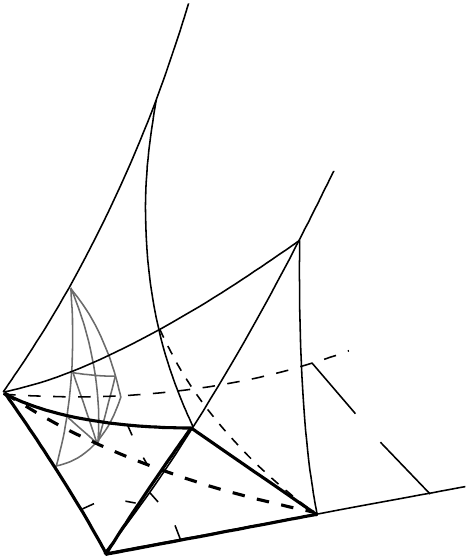}}\\
\multicolumn{2}{c}{(c)}
\end{tabular}
\caption{Fundamental polyhedra for $\Gamma^*$
($\gamma<0$)}\label{fp_disjoint}
\end{figure}
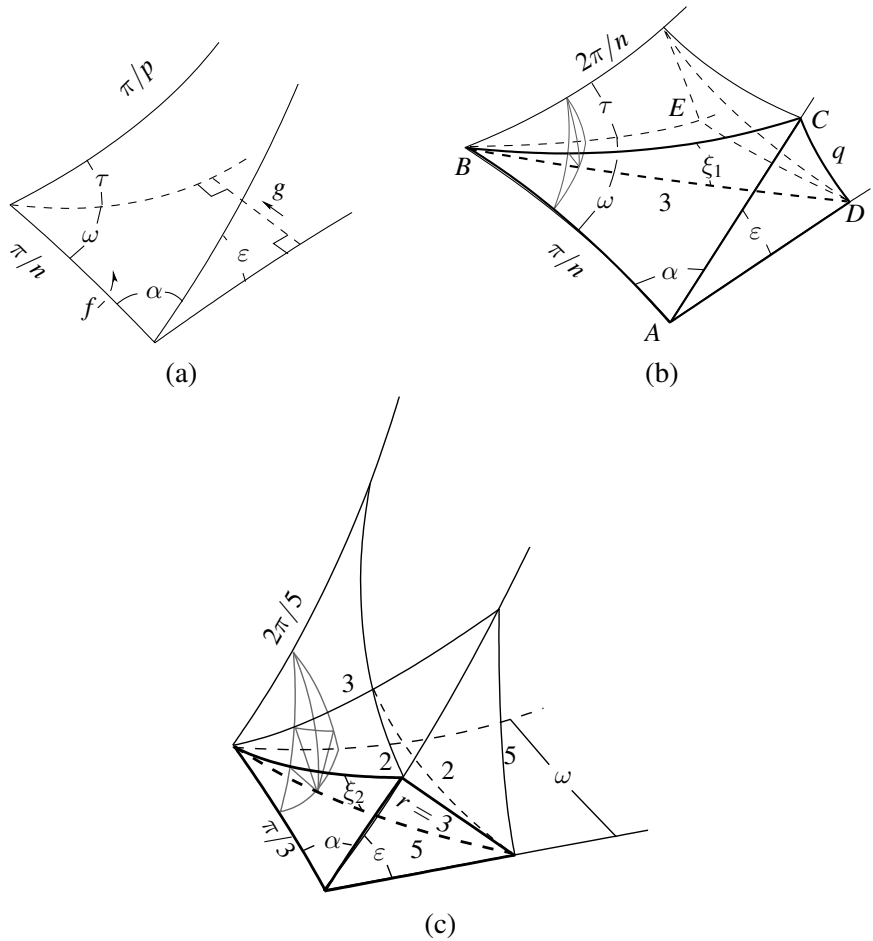

The other discrete groups appear only if $h$ is the square of a
primitive
elliptic element $x=R_\kappa R_\tau$, where $\kappa$ is
the bisector of the dihedral
angle of $\P$ made by $\alpha$ and $\tau$.
Fundamental polyhedra for $\Gamma^*$ corresponding to these two
series are
obtained by decomposing $\P$ into smaller polyhedra as follows
(see \cite{Kli90} for the proof).

{\bf 2.} \qua
Let $\Gamma$ be determined by the condition (2)(ii).
In this case, $n\geq 5$ is odd, the dihedral angle of $\P$
between
$\alpha$ and $\tau$ is $2\pi/n$,
and $\kappa$ and $\omega$ make a dihedral angle of $\pi/3$.
Hence, $\xi_1$, where $\xi_1=R_\kappa(\omega)$, and $\omega$ also
make a dihedral angle
of $\pi/3$, and $\xi_1$ and $\alpha$ are orthogonal. The planes
$\xi_1$ and $\varepsilon$ either intersect at an angle of
$\pi/q$, where
$q\in{\mathbb Z}$, $q\geq 4$, or are parallel or disjoint
($q=3$ is not included, because then $\varepsilon$ and $\tau$
must intersect).
One can show that if $y=R_\varepsilon R_{\xi_1}$, then
$y=hgfx^{-1}f$.
The polyhedron $\P(\omega,\varepsilon,\alpha,\xi_1)$ is a
fundamental polyhedron
for $\Gamma^*$.
For $q=4$ or $5$ and $n=5$,
$\P(\omega,\varepsilon,\alpha,\xi_1)$ is a compact tetrahedron.
It is denoted by $ABCD$ in \fullref{fp_disjoint}(b) and
shown by bold lines.

{\bf 3.} \qua
Let $\Gamma$ be determined by the condition (2)(iii).
In this case $n=3$ and the dihedral angle of $\P$ between
$\alpha$ and $\tau$
is $2\pi/5$. Denote $\xi_2=R_\kappa(\omega)$.
The planes $\kappa$ and $\omega$ make a dihedral angle of
$2\pi/5$ and, hence, $\xi_2$ and $\omega$ make an
angle of $\pi/5$. It can be shown that $\xi_2$ and $\alpha$ are
orthogonal. The planes $\varepsilon$ and $\xi_2$ either intersect
at an angle
of $\pi/r$, where $r\in{\mathbb Z}$, $r\geq 3$, or are parallel or
disjoint.
In this case $z=R_\varepsilon R_{\xi_2}=hgf(x^{-1}f)^3$. The
polyhedron
$\P(\omega,\varepsilon,\alpha,\xi_2)$ is a fundamental polyhedron
for
$\Gamma^*$ (see \fullref{fp_disjoint}(c), where
$\P(\omega,\varepsilon,\alpha,\xi_2)$ is drawn for $r=3$).

\subsection{Parameters for discrete groups in the case of
disjoint axes}

Let 
$$
\mathcal{U}=\{u|u=i\pi/p, p\in{\mathbb Z}, p\geq 2\}\cup[0,+\infty).
$$
Define a function
$t \co \thinspace
\mathcal{U}\to\{2,3,\dots\}\cup\{\infty,\overline\infty\}$ 
as follows:
$$
t(u)=
\begin{cases}
p & \text{if } u=i\pi/p\\
\infty & \text{if } u=0\\
\overline\infty & \text{if } u>0.
\end{cases}
$$

The purpose of introducing the function $t(u)$ is to shorten
statements that involve parameters $(\beta,\beta',\gamma)$.
We use it in Theorems~\ref{param_disjoint}, \ref{param_intersect},
\ref{groups_disjoint}, \ref{groups_even}
and~\ref{groups_odd}.

Now we give a parameter version of \fullref{discr_disjoint}
with a proof. \fullref{param_disjoint} is new and did not
appear
before, however, we did use a similar technique in earlier
papers.

\begin{theorem}\label{param_disjoint}
Let $(\Gamma;f,g)$ be an $\mathcal{RP}$ group with
$\beta=-4\sin^2(\pi/n)$, where $n\geq 3$ is an integer,
$\beta'\in(0,+\infty)$ and $\gamma\in(-\infty,0)$.
Then $\Gamma$ is discrete if and only if
one of the following holds:
\begin{enumerate}
\item
$\gamma=-4\cosh^2u$, where
$u\in\mathcal{U}$ and $t(u)\geq 3$;
\item
$n\geq 5$, $(n,2)=1$, $\gamma=-(\beta+2)^2$ and
$\beta'=4(\beta+4)\cosh^2u-4$, where
$u\in\mathcal{U}$ and $t(u)\geq 4$ or
\item
$\beta=-3$, $\gamma=(\sqrt{5}-3)/2$ and
$\beta'=2(7+3\sqrt{5})\cosh^2u-4$, where
$u\in\mathcal{U}$ and $t(u)\geq 3$.
\end{enumerate}
\end{theorem}

\begin{proof}
$\beta=-4\sin^2(\pi/n)$, where $n\in {\mathbb Z}$ and $n\geq 3$,
if and only if $f$ is a primitive
elliptic element of order $n\geq 3$, 
and $\beta'\in(0,+\infty)$ if and only if
$g$ is hyperbolic. Since $n\geq 3$ and $\gamma\in(-\infty,0)$,
$\Gamma$ is non-elementary and the axes of $f$ and $g$ are
disjoint by Klimenko and Kopteva \cite[Theorem 1]{KK02}.
So the hypotheses of \fullref{param_disjoint}
and \fullref{discr_disjoint} are equivalent.

Let us find explicit values of $\beta'$ and $\gamma$ for each of
the
discrete groups from part (2) of \fullref{discr_disjoint}.
The idea is to use the fundamental polyhedra
described in \fullref{construction_disjoint}.
Since $\gamma=\tr[f,g]-2$ and $h$
is a square root of $[f,g]$,
it is not difficult to get conditions on~$\gamma$.

We start with (2)(i) in \fullref{discr_disjoint}.
The element $h=R_\alpha R_\tau$ is hyperbolic if and only if the
planes
$\alpha$ and $\tau$ (see \fullref{fp_disjoint}(a)) are
disjoint in $\overline{{\mathbb H}^3}$.
Therefore, $\tr[f,g]=\tr h^2=-2\cosh(2d)$,
where $d$ is the hyperbolic distance between $\alpha$ and $\tau$.
Here $\tr[f,g]$ must be negative, because $\gamma$
is negative for all values of $\beta$ and $\beta'$ that satisfy
the hypotheses
of the theorem.

The element $h$ is parabolic if and only if $[f,g]$ is parabolic
and
if and only if $\tr[f,g]=-2$ which is equivalent to $\gamma=-4$
($\tr[f,g]=2$ would give $\gamma=0$).

Thus, $h$ is hyperbolic or parabolic if and only if
\begin{equation}\label{hyppar}
\gamma=\tr[f,g]-2=-2\cosh(2d)-2=-4\cosh^2d, \qua d\geq 0.
\end{equation}
Now suppose that $h$ is an elliptic element with rotation angle
$\phi$,
where $\phi/2=\pi/p<\pi/2$ is the dihedral angle of
$\P(\omega,\varepsilon,\alpha,\tau)$ made by
$\alpha$ and $\tau$. Then $[f,g]=h^2$ is also elliptic with
rotation angle
$2\phi$. 
Note that there is another square root $\overline h$ of
the elliptic commutator $[f,g]$ which has rotation angle
$\overline \phi=\pi+\phi$. One of the two angles
$\phi$ and $\overline \phi$ is a solution to
$\tr[f,g]=2\cos\theta$, the other one to 
$\tr[f,g]=-2\cos\theta$. Since $\tr[f,g]$ depends on $\theta$
continuously and we know that $\tr[f,g]$ must approach
$-2$ as $\phi\to 0$ (geometrically it means that
 $[f,g]$, and $h$, approaches a parabolic element as soon as
the dihedral angle $\pi/p$ above approaches $0$),
we conclude that $\phi$ 
is a solution to the second equation, 
that is, $\tr[f,g]=-2\cos\phi$.

On the other hand, if $\tr[f,g]\in(-2,2)$ is given, we can use
the formula
$\tr[f,g]=-2\cos\phi$, $0<\phi<\pi$, to determine the rotation
angle $\phi$
of the element $h$ from \fullref{discr_disjoint}.

Hence, $h$ is a primitive elliptic element of order $p$
($p\geq 3$),
that is,
$\phi=2\pi/p$, if and only if
\begin{equation}\label{primell}
\gamma=\tr[f,g]-2=-2\cos(2\pi/p)-2=-4\cos^2(\pi/p), \qua p\in{\mathbb
Z}, \qua p\geq 3.
\end{equation}
Now we can combine the formulas \eqref{hyppar} and
\eqref{primell} for $\gamma$ and write them as
$$\gamma=-4\cosh^2u, \text{ where } u\in\mathcal{U} \text{ and } 
t(u)\geq 3.$$

It is clear that for the groups from Item (2)(i) of
\fullref{discr_disjoint}, we have no further restrictions on
$n$ and $\beta'$.
So, (2)(i) of \fullref{discr_disjoint} is equivalent to
part 1 of \fullref{param_disjoint}.

Now consider (2)(ii) of \fullref{discr_disjoint}.
Here $n\geq 5$ is odd and $h$ is the square of a
primitive
elliptic element of order~$n$ (that is, $\phi=4\pi/n$)
if and only if $n\geq 5$, $(n,2)=1$ and
$\gamma=-4\cos^2(2\pi/n)=-(\beta+2)^2$.

So it remains to specify $\beta'$ for (2)(ii). Now $\beta'$
depends on
the order of the element $y$ defined in \fullref{discr_disjoint}.
Since $g=R_\tau R_\varepsilon$, we have
$\beta'=\tr^2g-4=4\sinh^2T$, where $T$ is the
distance between the planes $\varepsilon$ and $\tau$.

Let us show how to calculate $T$ for the case $n=5$ and $4\leq q<\infty$.
Since the link of the vertex $B$ in \fullref{fp_disjoint}(b)
is a spherical triangle, we get
$$
\cos \angle ABE=\frac{\cos(2\pi/n)}{\sin(\pi/n)} \text{ and }
\cos\angle ABC=\frac{1}{2\sin(\pi/n)}.
$$
Further, since $\triangle ABC$ is a right triangle with
$\angle ACB=\pi/q$, we have
$$
\cosh AB=\frac{\cos\angle ACB}{\sin\angle ABC}=
\frac{2\cos(\pi/q)\sin(\pi/n)}{\sqrt{4\sin^2(\pi/n)-1}}.
$$
Since $T$ is the length of the common perpendicular to $BE$ and $AD$,
we can now calculate $\cosh T$ from the plane $\omega$:
$$
\cosh T=\sin \angle ABE\cdot\cosh AB.
$$
Since
\begin{align*}
\sin^2\angle ABE&=\frac{\sin^2(\pi/n)-\cos^2(2\pi/n)}{\sin^2(\pi/n)}
=\frac{(\sin^2(\pi/n)-1)(1-4\sin^2(\pi/n))}{\sin^2(\pi/n)}\\
&=\frac{\cos^2(\pi/n)(4\sin^2(\pi/n)-1)}{\sin^2(\pi/n)},
\end{align*}
we get that
$$
\cosh^2T=4\cos^2(\pi/n)\cos^2(\pi/q)=(\beta+4)\cos^2(\pi/q).
$$
Hence, $\beta'=4\sinh^2T=4(\beta+4)\cos^2(\pi/q)-4$.

Analogous calculations can be done for the other cases
(when $n>5$ or $q\geq\infty$). We obtain
that for the groups from Item (2)(ii),
$$
\beta'=\begin{cases}
4(\beta+4)\cos^2(\pi/q)-4 & \text{if } 4\leq q<\infty\\
4(\beta+4)-4 & \text{if }  q=\infty\\
4(\beta+4)\cosh^2d_1-4 & \text{if }  q=\overline\infty,
\end{cases}
$$
where $d_1$ is the distance between $\varepsilon$ and $\xi_1$ if
they
are disjoint
and $\pi/q$ is the angle between $\varepsilon$ and $\xi_1$ if
they
intersect.
Hence, $\beta'$ can be written in general form as follows:
$$
\beta'=4(\beta+4)\cosh^2u-4 \text{ where } 
u\in\mathcal{U} \text{ and } t(u)\geq 4.$$
Finally, for the groups from Item (2)(iii) of \fullref{discr_disjoint},
we have $n=3$ and $\phi=4\pi/5$
and therefore
$$\beta = -3 \text{ and } 
\gamma=-4\cos^2(2\pi/5)=(\sqrt{5}-3)/2.$$
Moreover, we can calculate
$$
\beta'=
\begin{cases}
2(7+3\sqrt{5})\cos^2(\pi/r)-4 & \text{if }  3\leq r<\infty\\
2(7+3\sqrt{5})-4 & \text{if }  r=\infty\\
2(7+3\sqrt{5})\cosh^2d_2-4 & \text{if }  r=\overline\infty,
\end{cases}
$$
where $d_2$ is the distance between $\varepsilon$ and $\xi_2$ if
they
are disjoint
and $\pi/r$ is the angle between $\varepsilon$ and $\xi_2$ if
they
intersect and, hence,
$$\beta'=2(7+3\sqrt{5})\cosh^2u-4, \text{ where }  
u\in\mathcal{U} \text{ and } t(u)\geq 3.$$
\end{proof}

\subsection{Geometric description of discrete groups for the case
of intersecting axes}\label{construction_intersect}

Now we consider $\Gamma=\langle f,g \rangle$
with $f$ primitive elliptic of order $n>2$ and $g$ hyperbolic
with {\em non-orthogonally intersecting axes\/}.
In Klimenko--Kopteva \cite{KK02,KK04}, 
criteria for discreteness of such
groups were found for $n$ even and odd, respectively.
In this section we recall the criteria in terms of parameters
and remind the construction of a fundamental polyhedron for
each discrete group $\Gamma^*$.

\begin{theorem}[{\rm \cite{KK02} and \cite{KK04}}]\label{param_intersect}
Let $(\Gamma;f,g)$ be an $\mathcal{RP}$ group with
$\beta=-4\sin^2(\pi/n)$, where $n\geq 3$ is an integer, $\beta'\in(0,\infty)$ and
$\gamma\in(0,-\beta\beta'/4)$.
Then $\Gamma$ is discrete if and only if
$(\beta,\beta',\gamma)$ is one of the triples listed in
\fullref{table_param}.
\end{theorem}

{\small
\begin{longtable}{cccc}
\caption{All parameters for discrete $\mathcal{RP}$ groups generated
by a primitive elliptic element $f$ of order $n\geq 3$ and a 
hyperbolic element $g$ whose axes intersect non-orthogonally.}
\setobjecttype{Table}\label{table_param}\\
\hline
\rule[-2ex]{0ex}{5ex} & $\beta=\beta(f)$ & $\gamma=\gamma(f,g)$ &
$\beta'=\beta(g)$ \\
\hline
\endfirsthead
\caption[]{(continued)}\\
\hline
\rule[-2ex]{0ex}{5ex} & $\beta=\beta(f)$ & $\gamma=\gamma(f,g)$ &
$\beta'=\beta(g)$ \\
\hline
\endhead
\multicolumn{4}{c}%
{\rule[-2ex]{0ex}{6ex}$n\geq 4$, $(n,2)=2$, \ $u,v\in \mathcal{U}$,
\ $1/n+1/t(u)<1/2$}\\
$P_1$
   & \rule[-2ex]{0ex}{6ex} $-4\sin^2\frac{\pi}n,\ n\geq 4$
   & $\displaystyle 4\cosh^2u+\beta$,
   & $\displaystyle \frac{4}\gamma \cosh^2v-\frac
{4\gamma}\beta$,\\
   & & $(t(u),2)=2$ &  $t(v)\geq 3$\\ \pagebreak
$P_2$ 
   & \rule[-2ex]{0ex}{6ex} $-4\sin^2\frac{\pi}n,\ n\geq 4$
   & $\displaystyle 4\cosh^2u+\beta$,
   & $\displaystyle \frac{4(\gamma-\beta)}\gamma
\cosh^2v-\frac{4\gamma}\beta$,\\
   & & $(t(u),2)=1$ & $t(v) \geq 3$\\
$P_3$
   & \rule[-2ex]{0ex}{6ex} $-2$
   & $2\cos(2\pi/m)$, $m\geq 5$,
   & $\gamma^2+4\gamma$\\
   & & $(m,2)=1$ & \\
\hline
\multicolumn{4}{c}%
{\rule[-2ex]{0ex}{6ex}$n\geq 3$, $(n,2)=1$, \ $u,v\in \mathcal{U}$,\
$1/n+1/t(u)<1/2$;}\\
\multicolumn{4}{c}%
{$\displaystyle
S=-2\frac{(\gamma-\beta)^2\cos\frac{\pi}{n}+\gamma(\gamma+\beta)}
{\gamma\beta}$, \
 $\displaystyle
T=-2\frac{(\beta+2)^2\cos\frac{\pi}{n}}{\beta+1}-2\frac{
(\beta^2+6\beta+4)}\beta$}\\
$P_4$
   & \rule[-2ex]{0ex}{7ex} $-4\sin^2\frac{\pi}n, \ n\geq 3$
   & $\displaystyle4\cosh^2u+\beta$,
   & $\displaystyle\frac 2\gamma(\cosh v-\cos\frac{\pi}{n})+S$,
\\
   & & $(t(u),2)=2$ & $t(v)\geq 2$\\
$P_5$
   & \rule[-2ex]{0ex}{6ex} $-4\sin^2\frac{\pi}n, \ n\geq 3$
   & $\displaystyle4\cosh^2u+\beta$,
   & $\displaystyle{\frac{2(\gamma-\beta)}{\gamma}\cosh v+S}$,\\
   & & $(t(u),2)=1$ & $t(v)\geq 2$\\
$P_6$
   & \rule[-2ex]{0ex}{6ex} $-4\sin^2\frac{\pi}n, \ n\geq 7$
   & $(\beta+4)(\beta+1)$
   & $\displaystyle{\frac{2(\beta+2)^2}{\beta+1}\cosh v+T, \
t(v)\geq 2}$\\
$P_7$
   & \rule[-2ex]{0ex}{6ex} $\displaystyle-4\sin^2\frac{\pi}n$,
   & $\beta+3$
   & $\displaystyle\frac2{\beta}
   \left((\beta-3)\cos\frac{\pi}n-2\beta-3\right)$\\
   & $n\geq 5$, $(n,3)=1$ & & \\
$P_8$
   & \rule[-2ex]{0ex}{6ex} $\displaystyle-4\sin^2\frac{\pi}n$,
   & $2(\beta+3)$
   &
$\displaystyle-\frac6{\beta}\left(2\cos\frac{\pi}
n+\beta+2\right)$\\
   & $n\geq 5$, $(n,3)=1$ & & \\
$P_9$
   & \rule[-2ex]{0ex}{6ex} $-3$
   & $2\cos(2\pi/m)-1$,
   &
$\displaystyle{\frac2{\gamma}\left(\gamma^2+2\gamma+2\right)}$\\
   & & $m\geq 7$, $(m,2)=1$ & \\
$P_{10}$
   & \rule[-2ex]{0ex}{6ex} $-3$
   & $2\cos(2\pi/m)-1$,
   & $\gamma^2+4\gamma$\\
   & & $m\geq 8$, $(m,6)=2$ & \\
$P_{11}$
   & \rule[-2ex]{0ex}{6ex} $-3$
   & $2\cos(2\pi/m)$,
   & $2\gamma$\\
   & & $m\geq 7$, $(m,4)\leq 2$ & \\
$P_{12}$ & \rule[-1ex]{0ex}{4ex} $-3$ & $(\sqrt{5}+1)/2$ &
$\sqrt{5}$\\
$P_{13}$ & \rule[-1ex]{0ex}{4ex} $-3$ & $(\sqrt{5}-1)/2$ &
$\sqrt{5}$\\
$P_{14}$ & \rule[-1ex]{0ex}{4ex} $-3$ & $(\sqrt{5}-1)/2$ &
$\sqrt{5}-1$\\
$P_{15}$ & \rule[-1ex]{0ex}{4ex}$(\sqrt{5}-5)/2$ &
$(\sqrt{5}-1)/2$ & $\sqrt{5}$\\
$P_{16}$ & \rule[-1ex]{0ex}{4ex}$(\sqrt{5}-5)/2$ &
$(\sqrt{5}-1)/2$ & $(3\sqrt{5}-1)/2$\\
$P_{17}$ & \rule[-1ex]{0ex}{4ex}$(\sqrt{5}-5)/2$ &
$(\sqrt{5}-1)/2$ & $3(\sqrt{5}+1)/2$\\
$P_{18}$ & \rule[-1ex]{0ex}{4ex}$(\sqrt{5}-5)/2$ &
$(\sqrt{5}+1)/2$ & $3(\sqrt{5}+1)/2$\\
$P_{19}$ & \rule[-1ex]{0ex}{4ex}$(\sqrt{5}-5)/2$ & $\sqrt{5}+2$ &
$(5\sqrt{5}+9)/2$\\
\hline
\end{longtable}}

\begin{rem}
Note that if a formula in \fullref{table_param} involves
$u\in\mathcal{U}$
such that $(t(u),2)=1$, then $t(u)$ is finite and odd,
while for $u\in\mathcal{U}$ with $(t(u),2)=2$, $t(u)$ can be not
only finite (and even),
but $\infty$ or $\overline\infty$, which implies that the formula
is applicable also to $u\geq 0$.
In general, if $(m,k)<k$, then $m$ is finite.
\end{rem}

Let $f$ and $g$ be as in \fullref{param_intersect}, that is,
let $f$ be a primitive elliptic element of order $n\geq 3$, $g$
be a hyperbolic element and let their axes intersect
non-orthogonally. Let $\omega$ be the
plane containing $f$ and $g$, and let $e$ be a half-turn whose
axis is orthogonal to $\omega$ and passes through the point of
intersection of $f$ and $g$.

Again,  we define two finite
index extensions of $\Gamma=\langle f,g \rangle$ as follows:
$\widetilde{\Gamma}=\langle f,g,e\rangle$ and
$\Gamma^*=\langle f,g,e,R_\omega\rangle$.

Let $e_f$ and $e_g$ be half-turns such that $f=e_fe$ and
$g=e_ge$.
The lines $e_f$ and $e$ lie in a plane, denote it by
$\varepsilon$, and
intersect at an angle of $\pi/n$; $\varepsilon$ and $\omega$ are
mutually orthogonal; $e_g$ is orthogonal to $\omega$ and
intersects
$g$.

Let $\alpha$ be a hyperbolic plane such that $f=R_\omega
R_\alpha$
and let $\alpha'=e_g(\alpha)$.
There exists a plane $\delta$ which is orthogonal to the planes
$\alpha$,
$\omega$ and $\alpha'$. The plane $\delta$ passes
through the
common perpendicular to $f$ and $e_g(f)$ orthogonally to
$\omega$. It
is clear that $e_g\subset\delta$.

From here on, we describe the cases of even $n$
and odd $n$ separately ($n$ is the order of the elliptic
generator $f$).

{\bf $n\geq 4$ is even.} \qua
Let $\P=\P(\alpha,\omega,\alpha',\delta,\varepsilon)$.
The polyhedron $\P$ can be compact or non-compact;
in \fullref{fp_even}(a), $\P$ is drawn as compact.

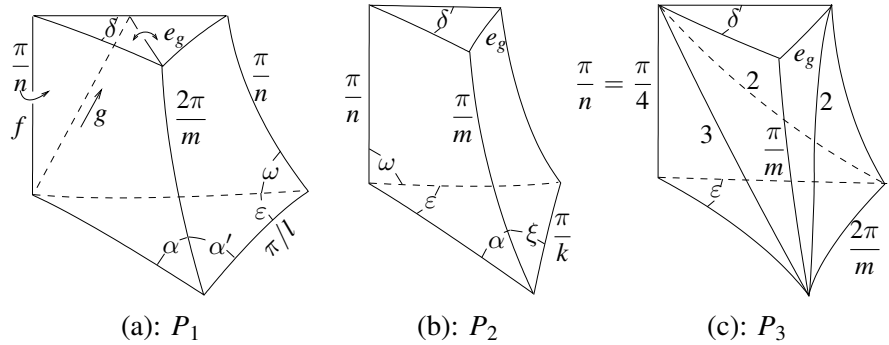
\begin{figure}[htp!]
\centering
\begin{tabular}{ccc}
\labellist
\small
\pinlabel {$\dfrac{\pi}{n}$} at 5 115
\pinlabel {$f$} at 5 85
\pinlabel {$g$} at 45 90
\pinlabel {$\delta$} at 50 135
\pinlabel {$e_g$} at 82 130 
\pinlabel {$\dfrac{2\pi}{m}$} at 90 90 
\pinlabel {$\dfrac{\pi}{n}$} at 125 110
\pinlabel {$\omega$} at 130 62
\pinlabel {$\varepsilon$} at 125 42
\pinlabel {$\alpha$} at 80 25
\pinlabel {$\alpha'$} at 105 27
\pinlabel {\turnbox{45}{$\pi/l$}} at 130 20
\endlabellist
\includegraphics[width=4 cm]{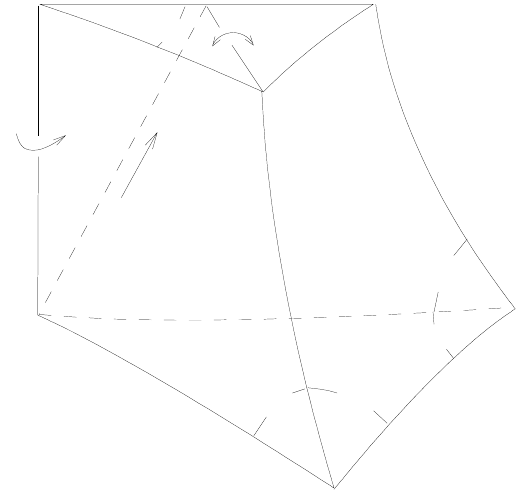} &
\labellist
\small
\pinlabel {$\dfrac{\pi}{n}$} at 5 100
\pinlabel {$\delta$} at 50 140
\pinlabel {$e_g$} at 77 128 
\pinlabel {$\dfrac{\pi}{m}$} at 60 90 
\pinlabel {$\omega$} at 22 67
\pinlabel {$\varepsilon$} at 44 48
\pinlabel {$\alpha$} at 77 28
\pinlabel {$\xi$} at 95 32
\pinlabel {$\dfrac{\pi}{k}$} at 110 30
\endlabellist
\includegraphics[width=3.05 cm]{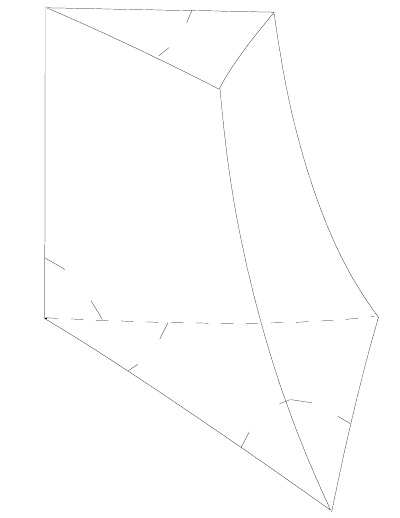} &
\labellist
\small
\pinlabel {$\dfrac{\pi}{n}= \dfrac{\pi}{4}$} at 5 110
\pinlabel {$\delta$} at 67 145
\pinlabel {$e_g$} at 104 124
\pinlabel {$\dfrac{2\pi}{m}$} at 135 26
\pinlabel {$\dfrac{\pi}{m}$} at 88 76
\pinlabel {$\varepsilon$} at 58 55
\pinlabel {$2$} at 77 113
\pinlabel {$2$} at 115 102
\pinlabel {$3$} at 53 85
\endlabellist
\includegraphics[width=3.8 cm]{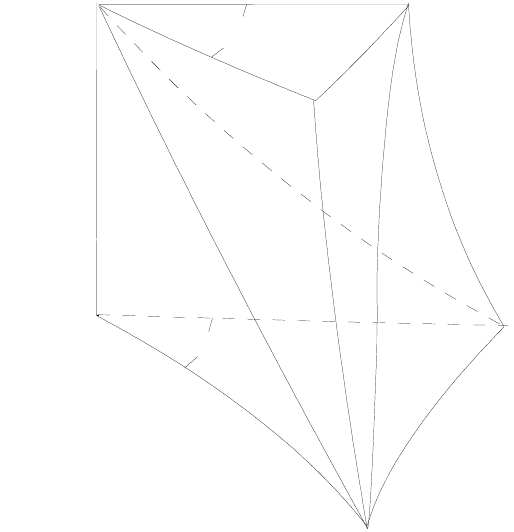}\\
(a): $P_1$ & (b): $P_2$ & (c): $P_3$
\end{tabular}
\caption{Fundamental polyhedra for $\Gamma^*$ in
the case of even~$n$ ($0<\gamma<-\beta\beta'/4$)}\label{fp_even}
\end{figure}

The polyhedron $\P$ has five right dihedral angles;
the dihedral angles formed by $\omega$ with $\alpha$ and
$\alpha'$
equal $\pi/n$.
The planes $\alpha$ and $\alpha'$ can either intersect
or be parallel or disjoint; the same is true for
$\varepsilon$ and $\alpha'$.
Denote the angle between $\varepsilon$ and $\alpha'$
by $\pi/\ell$, where
$\ell\in(2,\infty)\cup\lbrace\infty,\overline\infty\rbrace$
and denote the angle between $\alpha$ and $\alpha'$
by $2\pi/m$, where
$m\in(2,\infty)\cup\lbrace\infty,\overline\infty\rbrace$,
$1/n+1/m<1/2$.

For each triple of parameters with $n$ even in
\fullref{table_param},
we know (from the paper \cite{KK02}) how a fundamental polyhedron for
$\Gamma^*$ looks like,
and we describe all such polyhedra below.

$\mathbf P_1$. \qua
$\P$ is a fundamental polyhedron for $\Gamma^*$ if and only if
$m\in{\mathbb Z}\cup\{\infty,\overline\infty\}$, $m$ is even
($1/m+1/n<1/2$) and
$\ell\in{\mathbb Z}\cup\{\infty,\overline\infty\}$
($\ell\geq 3$). In terms of the function $t$,
$m=t(u)$ and $\ell=t(v)$ (cf \fullref{table_param}).

$\mathbf P_2$. \qua
Note that in this case
$m=t(u)$ is finite and odd.
Let $\xi$ be the bisector of the dihedral angle of $\P$ at the
edge
$\alpha\cap \alpha'$. It is clear that $\xi$ passes through
$e_g$ and is orthogonal to $\omega$.
The polyhedron $\P(\alpha,\delta,\xi,\varepsilon,\omega)$ is
bounded by
reflection planes of $\Gamma^*$ (see \fullref{fp_even}(b))
and,
therefore, it is a fundamental polyhedron for $\Gamma^*$ if and
only if
$\xi$ and $\varepsilon$
intersect at an angle of $\pi/k$, where $k\geq 3$, or
are parallel or disjoint ($k=\infty$ or $k=\overline\infty$,
respectively).
In \fullref{table_param}, $k=t(v)$ for the parameters $P_2$.

$\mathbf P_3$. \qua In this case $n=4$ and the dihedral angle
of $\P(\alpha,\delta,\xi,\varepsilon,\omega)$ at the edge
$\xi\cap\varepsilon$ is $2\pi/m$, where $m=t(u)$ is odd,
$5\leq m<\infty$.
The polyhedron $\P(\alpha,\delta,\xi,\varepsilon,\omega)$ is
decomposed by reflection planes of $\Gamma^*$ into three
(possibly infinite volume)
tetrahedra $T[2,2,4;2,3,m]$, each of which is a
fundamental polyhedron for $\Gamma^*$ (see
\fullref{fp_even}(c)).

{\bf $n\geq 3$ is odd.} \qua
Denote $e_1=f^{(n-1)/2}e$.
Note that $e_1$ makes angles of $\pi/(2n)$ with $\alpha$
and $\omega$.

We can forget about the plane $\varepsilon$, because now we need
another plane,
denote it by $\zeta$, for the
construction of a fundamental polyhedron for $\Gamma^*$.
To construct $\zeta$ we use an auxiliary plane $\kappa$ that
passes
through $e_1$ orthogonally to $\alpha'$. The plane $\zeta$ then
passes through $e_1$ orthogonally to $\kappa$.
(Note that $\zeta$ is not orthogonal
to each of the planes $\alpha$ and $\omega$ if $m\neq 2n$.)
In fact, the planes $\zeta$ and $\alpha'$ can either
intersect or be parallel or disjoint. Note that if
$\zeta\cap \alpha'\not=\emptyset$ then
$e_1$ is orthogonal to $\zeta\cap \alpha'$.
Let $\P=\P(\alpha,\omega,\alpha',\delta,\zeta)$.
In \fullref{fp_odd}(a), $\P$ is drawn for the compact case.

Consider the dihedral angles of $\P$. The angles between $\delta$
and
$\omega$, $\delta$ and $\alpha$, $\delta$ and $\alpha'$ are all
of
$\pi/2$; the angles formed by $\omega$ with $\alpha$ and
$\alpha'$
equal $\pi/n$; since $\zeta$ passes through $e_1$, which is orthogonal
to $f$, the sum of the angles $\phi$ and $\psi$ 
formed by $\zeta$ with $\alpha$ and $\omega$,
respectively, equals $\pi$. The planes $\alpha$ and
$\alpha'$ can either intersect or be parallel or disjoint. The
same is true for $\zeta$ and $\alpha'$. Denote the angle between
$\alpha$ and $\alpha'$ by $2\pi/m$ and the angle between
$\alpha'$ and $\zeta$ by $\pi/(2\ell)$.

Fundamental polyhedra for groups $\Gamma^*$
for all triples of parameters with
odd $n$ from \fullref{table_param} were constructed in
\cite{KK04}.
Now we describe them.

\begin{figure}[ht!]
\centering
\begin{tabular}{cc}
\labellist
\small
\pinlabel {$\dfrac{\pi}{n}$} at 5 115
\pinlabel {$f$} at 5 85
\pinlabel {$\phi$} at 30 43
\pinlabel {$\psi$} at 47 70
\pinlabel {$\alpha$} at 28 117 
\pinlabel {$e_g$} at 64 130
\pinlabel {$\delta$} at 86 130
\pinlabel {$\alpha'$} at 112 115
\pinlabel {$\dfrac{\pi}{n}$} at 146 116
\pinlabel {$\omega$} at 129 83
\pinlabel {$\zeta$} at 139 53
\pinlabel {$e_1$} at 76 36
\pinlabel {$\dfrac{2\pi}{m}$} at 80 80 
\pinlabel {\turnbox{30}{$\pi/2l$}} at 140 27
\endlabellist
\includegraphics[width=4.5 cm]{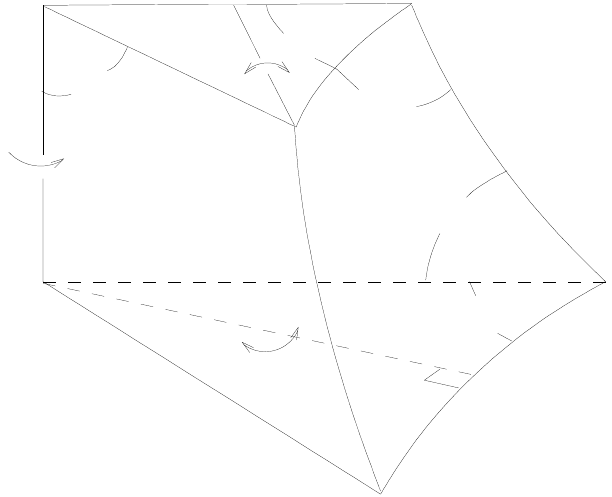}\qquad &
\qquad 
\labellist
\small
\pinlabel {$\dfrac{\pi}{n}$} at 3 70
\pinlabel {$\phi_1$} at 34 23
\pinlabel {$\alpha$} at 26 88 
\pinlabel {$e_g$} at 83 85 
\pinlabel {$\delta$} at 74 103
\pinlabel {$\dfrac{\pi}{m}$} at 69 64
\pinlabel {$\omega$} at 26 55
\pinlabel {$\zeta_1$} at 103 38
\pinlabel {$e_1$} at 72 27 
\pinlabel {$\xi$} at 93 89
\pinlabel {$\psi_1$} at 53 50
\pinlabel {\turnbox{45}{$\pi/2k$}} at 115 15
\endlabellist
\includegraphics[width=4.5 cm]{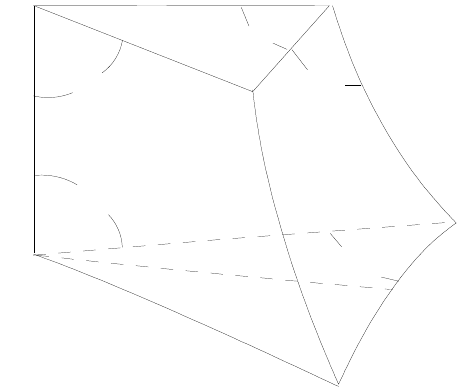}\\
(a):  $P_4$\qquad & \qquad(b): $P_5$\\
\end{tabular}
\caption{}\label{fp_odd}
\end{figure}

$\mathbf P_4$. \qua
$\P$ itself is a fundamental polyhedron for $\Gamma^*$ if and
only if
$m$ is even ($1/m+1/n<1/2$), $m=\infty$ or $m=\overline\infty$,
and
$\ell\in{\mathbb Z}\cup\{\infty,\overline\infty\}$, $\ell\geq 2$.
In terms of the function $t$,
$m=t(u)$ and $\ell=t(v)$ (cf \fullref{table_param}).

$\mathbf P_5$. \qua
Let $\xi$ be the bisector of the dihedral angle of $\P$ at the
edge
$\alpha\cap \alpha'$. Clearly, $\xi$ passes through
$e_g$ orthogonally to $\omega$ and $\delta$.
Construct a plane $\zeta_1$ in a similar way as $\zeta$ above
(now $\xi$ plays the role of $\alpha'$).
The polyhedron $\Q=\P(\alpha,\delta,\xi,\zeta_1,\omega)$ is a
fundamental
polyhedron for $\Gamma^*$ (see \fullref{fp_odd}(b))
if and only if $m$ is odd and $\xi$ and $\zeta_1$
make an angle of $\pi/(2k)$, where $k\geq 2$ is an integer,
$\infty$ or $\overline\infty$.

$\mathbf P_6$. \qua
In this case, the dihedral angle of $\Q$ at the edge
$\alpha\cap\xi$
equals $2\pi/n$ (ie $m=n/2$), $n\geq 7$ is odd.
Let $\rho$ be the bisector of this dihedral angle and
let $\tau=R_\rho(\omega)$. The bisector $\rho$ makes an angle
of $\pi/3$ with $\omega$ and, therefore, so does $\tau$.
It is clear that then $\tau$ is orthogonal to $\alpha$
(in $\delta$, 
we have one of Knapp's triangles with one non-primitive and two 
primitive angles leading to a discrete group \cite{Kna68}).
Construct a plane $\zeta_2$ similarly to the planes $\zeta$ and
$\zeta_1$
above (but using $\tau$).
The polyhedron $\P(\alpha,\delta,\omega,\tau,\zeta_2)$
(see \fullref{p1_hyp}(a), where we show also a part of the
plane
$\delta$)
is a fundamental polyhedron for $\Gamma^*$ if and only if
$\zeta_2$ and $\tau$ intersect at an angle of $\pi/(2k)$, where
$k\geq 2$ is an integer, or are parallel or disjoint ($k=\infty$
or $k=\overline\infty$, respectively). In
\fullref{table_param},
$t(v)$ corresponds to $k$.

$\mathbf P_9$. \qua
The dihedral angles of $\Q$ at the edges $\alpha\cap\xi$
and $\zeta_1\cap\xi$ equal $\pi/m$, $m$ is odd. The plane
$\zeta_1$ makes dihedral angles of $2\pi/3$ and $\pi/3$ with
$\alpha$ and $\omega$, respectively. Let $\sigma$ be the bisector
of the dihedral angle of $\Q$ at $\alpha\cap\zeta_1$. It is clear
that
$\sigma$ is orthogonal to $\omega$.
$\P(\alpha,\omega,\xi,\delta,\mu)$, where $\mu$ is the plane
that passes through $\sigma\cap\omega$ orthogonally to $\alpha$,
is a fundamental polyhedron for $\Gamma^*$
(see \fullref{p1_hyp}(b)).
The dihedral angle of $\P(\alpha,\omega,\xi,\delta,\mu)$ at
$\mu\cap\omega$ equals $\pi/4$.

\begin{figure}[ht!]
\centering
\begin{tabular}{cc}
\labellist
\small
\pinlabel {$\alpha$} at 15 70
\pinlabel {$\dfrac{\pi}{n}$} at 3 50
\pinlabel {$\omega$} at 16 39 
\pinlabel {$\psi_2$} at 35 37
\pinlabel {$e_1$} at 49 19
\pinlabel {$\phi_2$} at 25 15
\pinlabel {\turnbox{45}{$\pi/2k$}} at 78 10
\pinlabel {$\zeta_2$} at 68 28 
\pinlabel {$\tau$} at 59 57 
\pinlabel {$\delta$} at 43 80
\pinlabel {$\dfrac{\pi}{3}$} at 77 58
\pinlabel {$e_g$} at 76 82
\endlabellist
\includegraphics[width=5.2 cm]{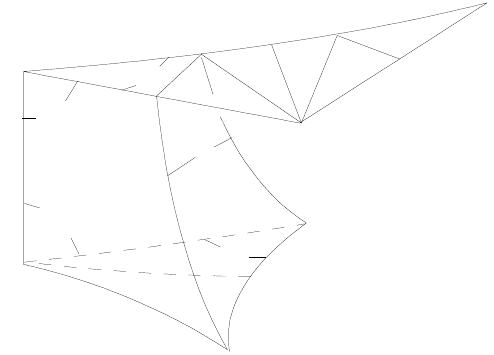}\qquad &
\qquad 
\labellist
\small
\pinlabel {\turnbox{-70}{$\pi/m$}} at 15 82
\pinlabel {$\dfrac{\pi}{3}$} at 4 61
\pinlabel {$\xi$} at 26 98
\pinlabel {$2$} at 21 54
\pinlabel {$2$} at 47 70
\pinlabel {$4$} at 48 51
\pinlabel {$\sigma$} at 60 59 
\pinlabel {$2$} at 61 49 
\pinlabel {$2$} at 77 56
\pinlabel {$2$} at 81 26 
\pinlabel {$\zeta_1$} at 102 43
\pinlabel {\turnbox{-45}{$2\pi/3$}} at 23 15
\pinlabel {$e_1$} at 45 28
\pinlabel {$\alpha$} at 54 13
\pinlabel {\turnbox{10}{$\pi/3$}} at 86 50
\pinlabel {\turnbox{30}{$\pi/m$}} at 94 17
\endlabellist
\includegraphics[width=4.3 cm]{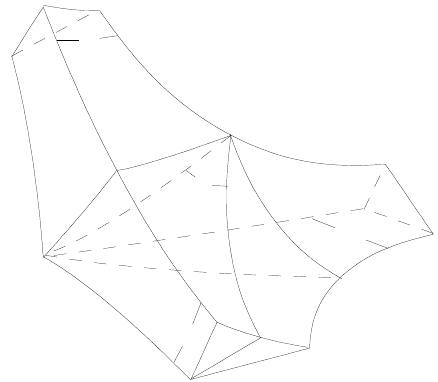}\\
(a): $P_6$\qquad &\qquad (b): $P_9$\\
\end{tabular}
\caption{}\label{p1_hyp}
\end{figure}

Fundamental polyhedra for remaining discrete groups $\Gamma^*$
are
obtained after decomposition of $\P$ (which is shown in \fullref{fp_odd}(a))
by the planes of reflections from~$\Gamma^*$,
that is, $(m,2)=2$, $m\geq 4$ and $\ell$ is fractional. We first
consider
the cases where $R_\zeta\in\Gamma^*$.

A compact convex polyhedron in ${\mathbb H}^3$ whose skeleton
is a trivalent graph is
uniquely determined by its dihedral angles up to isometry of
${\mathbb H}^3$, see Hodgson and Rivin \cite{HR93}.
Given $n$, $m$ and $\ell$, all the dihedral angles of the polyhedron
$\P\cup e_1(\P)$ are defined. Therefore, the dihedral angle $\phi$
of $\P$ at $\alpha\cap\zeta$ can be obtained.
So to determine a compact $\P$ it is sufficient
to indicate only $n$, $m$ and $\ell$, but we shall also give
the value of $\phi$ for convenience.
If $\P$ has infinite volume, but $\ell<\infty$ and $m<\infty$
(then $2/m+1/n+1/\ell < 1$),
$\P$ is also determined by the values of $n$, $m$ and $\ell$,
since we can cut off a compact polyhedron from $\P\cup e_1(\P)$
by a plane orthogonal to $\zeta$, $\alpha$, $\alpha'$ and a plane
orthogonal to $\zeta$, $\omega$, $\alpha'$.

There are no discrete groups for which $m=\infty$ or
$\ell\geq \infty$ except for those with parameters of type $P_4$.
When $m=\overline\infty$ we also indicate the distance $d$
between $\alpha$ and $\alpha'$ to determine~$\P$.
In fact, given $d$, one can find $\phi$, but we shall give
$\phi$ explicitly for convenience.

In all of these cases $R_\delta\not\in\Gamma^*$, so
we do not show $\delta$ (but indicate $e_g$) in figures in order
to simplify
the picture.
By the same reason we draw only those parts of the decomposition
(including $\omega$) that are important for the reconstruction of
the action
of $\Gamma^*$ and help to determine positions of $e_1$ and $e_g$.

\begin{figure}[ht!]
\centering
\begin{tabular}{cc}
\labellist
\small
\pinlabel {$2$} at 36 84
\pinlabel {$\pi/3$} at 2 65
\pinlabel {$r$} at 23 73
\pinlabel {$2$} at 46 73
\pinlabel {$\tau$} at 90 77
\pinlabel {$\kappa$} at 102 78
\pinlabel {$\pi/3$} at 125 91
\pinlabel {$s'$} at 48 62
\pinlabel {$u$} at 99 68
\pinlabel {$2$} at 23 52
\pinlabel {$4$} at 53 52
\pinlabel {$e_g$} at 68 51
\pinlabel {$s$} at 77 53
\pinlabel {$3$} at 97 55
\pinlabel {$\pi/3$} at 41 40
\pinlabel {$2$} at 99 39 
\pinlabel {$2$} at 110 43
\pinlabel {$\zeta$} at 118 45
\pinlabel {$\rho$} at 122 52
\pinlabel {$e_1$} at 54 26
\pinlabel {$2$} at 97 29
\pinlabel {\turnbox{-35}{$2\pi/3$}} at 23 22
\pinlabel {$t_1$} at 50 10
\pinlabel {\turnbox{50}{$2\pi/r$}} at 118 16
\endlabellist
\includegraphics[width=5.8 cm]{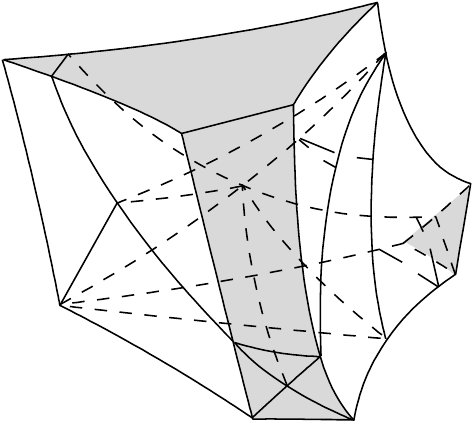}
& 
\labellist
\small
\pinlabel {$\alpha$} at 14 96
\pinlabel {$A$} at 77 114
\pinlabel {$\dfrac{\pi}{3}$} at 92 101
\pinlabel {$\dfrac{\pi}{3}$} at 4 88
\pinlabel {$5$} at 34 86
\pinlabel {$e_g$} at 45 78
\pinlabel {$5$} at 92 83
\pinlabel {$\omega$} at 107 85
\pinlabel {$V$} at 9 67
\pinlabel {$2$} at 39 57
\pinlabel {$2$} at 81 69
\pinlabel {$3$} at 94 64
\pinlabel {$\zeta$} at 119 68
\pinlabel {$\dfrac{4\pi}{5}$} at 18 47
\pinlabel {$2$} at 54 45
\pinlabel {$e_1$} at 77 43
\pinlabel {$C$} at 90 40 
\pinlabel {$\dfrac{\pi}{3}$} at 103 42
\pinlabel {$B$} at 127 53
\pinlabel {$3$} at 47 37
\pinlabel {$B'$} at 77 13
\endlabellist
\includegraphics[width=5.9 cm]{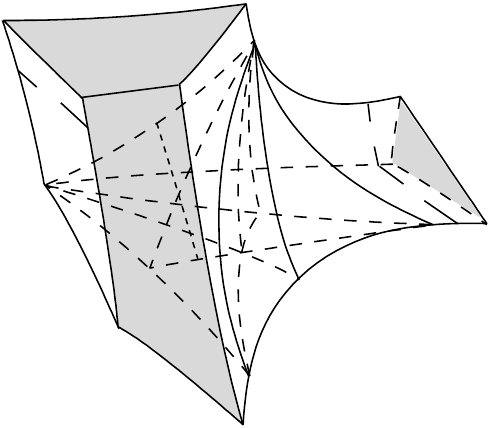}\\
(a): $P_{11}$ & (b): $P_{12}$ \\
\end{tabular}
\caption{}\label{gamma9-16}
\end{figure}

$\mathbf P_{11}$. \qua
$n=3$, $m=\overline\infty$, $\ell=r/4$, $(r,4)\leq 2$, $r\geq 7$,
$\phi=2\pi/3$
and $\cosh d=2\cos^2(\pi/r)-1/2$.
$\P(\alpha,\alpha',\omega,\zeta)$
is decomposed into tetrahedra $T=T[2,3,r;2,2,4]$ each of which
is a fundamental polyhedron for $\Gamma^*$
(\fullref{gamma9-16}(a)).
$\Gamma^*=G_T$.

$\mathbf P_{12}$. \qua
$n=3$, $m=\overline\infty$, $\ell=3/2$, $\phi=4\pi/5$
and $\cosh d=(3+\sqrt{5})/4$.
$\P(\alpha,\alpha',\omega,\zeta)$
is decomposed into tetrahedra $T=T[2,2,3;2,5,3]$.
A half of $T$ is a fundamental polyhedron for $\Gamma^*$
(\fullref{gamma9-16}(b)).
$\Gamma^*=\langle G_T,e_g\rangle$.

\begin{figure}[htp!]
\centering
\begin{tabular}{cc}
\labellist
\small
\pinlabel {$\dfrac{\pi}{3}$} at 3 70
\pinlabel {$\pi/5$} at 40 77
\pinlabel {$5$} at 62 78
\pinlabel {$\dfrac{\pi}{3}$} at 84 95
\pinlabel {$3$} at 21 50
\pinlabel {$e_g$} at 42 46
\pinlabel {$5$} at 64 61
\pinlabel {$3$} at 74 57
\pinlabel {$2$} at 81 62
\pinlabel {$2$} at 84 51
\pinlabel {$2$} at 92 42
\pinlabel {$2$} at 87 29
\pinlabel {\turnbox{-35}{$3\pi/5$}} at 19 21
\pinlabel {$e_1$} at 34 26
\pinlabel {$2$} at 47 19
\pinlabel {$2$} at 78 18
\pinlabel {\turnbox{30}{$2\pi/3$}} at 88 12
\endlabellist
\includegraphics[width=5.5 cm]{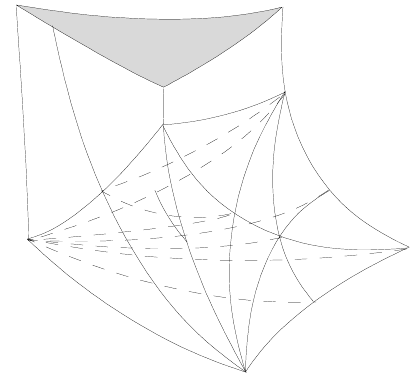} \qquad &
\qquad
\labellist
\small
\pinlabel {$\dfrac{\pi}{3}$} at 3 65
\pinlabel {$\rho$} at 25 61
\pinlabel {$2$} at 28 71
\pinlabel {$\dfrac{\pi}{5}$} at 52 66
\pinlabel {$4$} at 76 63
\pinlabel {$\dfrac{\pi}{3}$} at 100 83
\pinlabel {$\pi/3$} at 46 34
\pinlabel {$e_g$} at 59 40
\pinlabel {$3$} at 81 51
\pinlabel {$2$} at 93 48
\pinlabel {$e_1$} at 39 21
\pinlabel {\turnbox{-30}{$2\pi/3$}} at 26 13
\pinlabel {\turnbox{35}{$2\pi/5$}} at 89 11
\endlabellist
\includegraphics[width=5.0 cm]{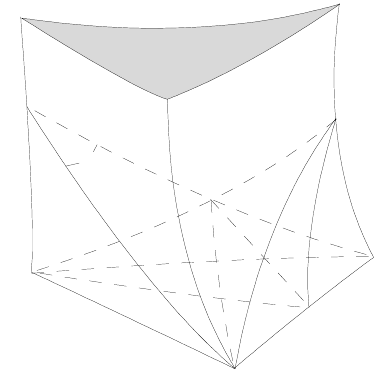}\\
(a): $P_{13}$\qquad & \qquad (b): $P_{14}$\\
\labellist
\small
\pinlabel {$\dfrac{\pi}{5}$} at 2 65
\pinlabel {$\dfrac{\pi}{5}$} at 96 80
\pinlabel {$e_g$} at 55 56
\pinlabel {\turnbox{10}{$4\pi/5$}} at 68 54
\pinlabel {$3$} at 55 44
\pinlabel {$2$} at 68 47
\pinlabel {$e_1$} at 37 35
\pinlabel {$2$} at 50 32
\pinlabel {$2$} at 73 39
\pinlabel {\turnbox{70}{$\pi/3$}} at 87 32
\pinlabel {\turnbox{-40}{$\pi/5$}} at 36 20
\pinlabel {$\dfrac{\pi}{2}$} at 60 25
\endlabellist
\includegraphics[width=5.5 cm]{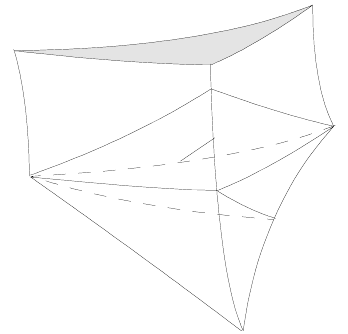} \qquad &
\qquad
\labellist
\small
\pinlabel {$\dfrac{\pi}{5}$} at 2 81
\pinlabel {$2$} at 28 75
\pinlabel {$e_g$} at 43 65
\pinlabel {$3$} at 58 75
\pinlabel {$5$} at 65 92
\pinlabel {$2$} at 71 73
\pinlabel {$\dfrac{\pi}{5}$} at 83 86
\pinlabel {\turnbox{-10}{$2\pi/3$}} at 78 49
\pinlabel {\turnbox{-55}{$\pi/3$}} at  25 28
\pinlabel {$e_1$} at 44 34
\pinlabel {\turnbox{-80}{$\pi/2$}} at 50 26
\pinlabel {$2$} at 68 34
\pinlabel {\turnbox{45}{$\pi/5$}} at 81 20
\endlabellist
\includegraphics[width=5.5 cm]{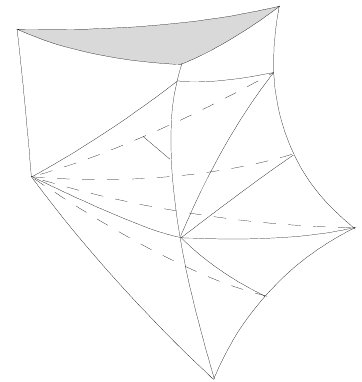}\\
(c): $P_{15}$\qquad & \qquad (d): $P_{17}$\\
\multicolumn{2}{c}{
\labellist
\small
\pinlabel {$\dfrac{\pi}{3}$} at 53 79
\pinlabel {$\dfrac{\pi}{5}$} at  89 80
\pinlabel {$\dfrac{\pi}{5}$} at 2 56
\pinlabel {$2$} at 15 58
\pinlabel {$2$} at 64 53
\pinlabel {$3$} at 77 58
\pinlabel {$5$} at 16 44
\pinlabel {$e_g$} at 49 38
\pinlabel {$2\pi/3$} at 25 34
\pinlabel {\turnbox{-30}{$\pi/3$}} at 20 19
\pinlabel {$e_1$} at 36 23
\pinlabel {$2$} at 66 20
\pinlabel {\turnbox{30}{$2\pi/5$}} at 88 20
\endlabellist
\includegraphics[width=5.3 cm]{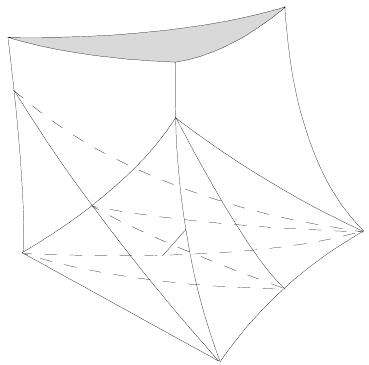}}\\
\multicolumn{2}{c}{(e): $P_{18}$}\\
\end{tabular}
\caption{}\label{gamma_central}
\end{figure}

$\mathbf P_{13}$. \qua
$n=3$, $m=10$, $\ell=3/2$, $\phi=3\pi/5$.
$\P(\alpha,\alpha',\omega,\zeta)$
is decomposed into tetrahedra $T=T[2,3,5;2,3,2]$.
A half of $T$ is a fundamental polyhedron for $\Gamma^*$
(\fullref{gamma_central}(a)).
$\Gamma^*=\langle G_T,e_g\rangle$.

$\mathbf P_{14}$. \qua
$n=3$, $m=10$, $\ell=5/4$, $\phi=2\pi/3$.
$\P(\alpha,\alpha',\omega,\zeta)$
is decomposed into tetrahedra $T=T[2,3,5;2,2,4]$ each of which
is a fundamental polyhedron for $\Gamma^*$
(\fullref{gamma_central}(b)).
$\Gamma^*=G_T$.

$\mathbf P_{15}$. \qua
$n=5$, $m=4$, $\ell=3/2$, $\phi=\pi/5$.
$\P(\alpha,\alpha',\omega,\zeta)$
is decomposed into tetrahedra $T=T[2,3,5;2,3,2]$.
A half of $T$ is a fundamental polyhedron for $\Gamma^*$
(\fullref{gamma_central}(c)).
$\Gamma^*=\langle G_T,e_g\rangle$.

$\mathbf P_{17}$. \qua
$n=5$, $m=4$, $\ell=5/2$, $\phi=\pi/3$.
$\P(\alpha,\alpha',\omega,\zeta)$
is decomposed into tetrahedra $T=T[2,3,5;2,2,5]$.
A half of $T$ is a fundamental polyhedron for $\Gamma^*$
(\fullref{gamma_central}(d)).
$\Gamma^*=\langle G_T,e_g\rangle$.

$\mathbf P_{18}$. \qua
$n=5$, $m=6$, $\ell=5/4$, $\phi=\pi/3$.
$\P(\alpha,\alpha',\omega,\zeta)$
is decomposed into tetrahedra $T=T[2,3,5;2,2,5]$.
A half of $T$ is a fundamental polyhedron for $\Gamma^*$
(\fullref{gamma_central}(e)).
$\Gamma^*=\langle G_T,e_g\rangle$.

\begin{figure}[htp!]
\centering
\begin{tabular}{cc}
\labellist
\small
\pinlabel {$3$} at 46 68
\pinlabel {$\dfrac{\pi}{3}$} at 71 73
\pinlabel {$\dfrac{\pi}{n}$} at 104 87
\pinlabel {$\dfrac{\pi}{n}$} at 3 55
\pinlabel {$2$} at 43 50
\pinlabel {$e_g$} at 57 55
\pinlabel {$e_1$} at 25 34
\pinlabel {\turnbox{30}{$\pi/2$}} at 33 25
\pinlabel {\turnbox{50}{$2\pi/n$}} at 92 35
\pinlabel {\turnbox{15}{$\pi/3$}} at 46 9
\endlabellist
\includegraphics[width=5.4 cm]{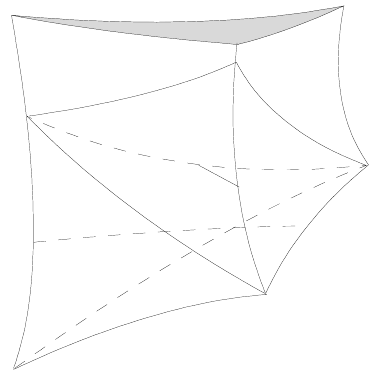} \qquad &
\qquad
\labellist
\small
\pinlabel {$\pi/n$} at 68 104
\pinlabel {$\sigma$} at 32 83
\pinlabel {$4$} at 41 76
\pinlabel {$2$} at 56 80
\pinlabel {$2$} at 77 79
\pinlabel {$\omega$} at 88 92
\pinlabel {$\alpha$} at 13 76
\pinlabel {$e_1$} at 15 61
\pinlabel {$\kappa$} at 94 80
\pinlabel {\turnbox{90}{{$\pi/n$}}} at 3 71
\pinlabel {\turnbox{-55}{\tiny{$2\pi/n$}}} at 28 72
\pinlabel {$L$} at 45 65
\pinlabel {$M$} at 58 68
\pinlabel {$2$} at 73 57
\pinlabel {\turnbox{55}{\tiny{$2\pi/4$}}} at 84 60
\pinlabel {$\alpha'$} at 96 60
\pinlabel {\turnbox{75}{{$\pi/n$}}} at 108 55
\pinlabel {$s$} at 44 50
\pinlabel {$e_g$} at 59 45
\pinlabel {$2$} at 92 46
\pinlabel {\turnbox{5}{\tiny{$\pi/n$}}} at 35 37
\pinlabel {\footnotesize{$\eta$}} at 30 31
\pinlabel {$K$} at 33 20
\pinlabel {\turnbox{-30}{{$2\pi/4$}}} at 37 16
\pinlabel {$L'$} at 89 4
\pinlabel {\turnbox{25}{\tiny{$2\pi/n$}}} at 92 6
\endlabellist
\includegraphics[width= 5.9 cm]{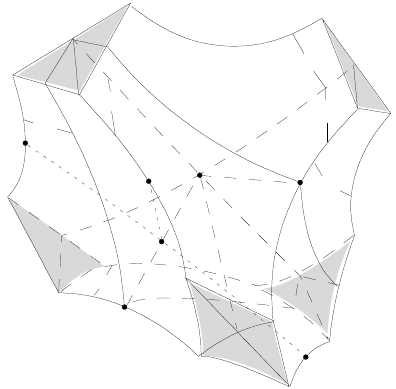}\\
(a): $P_{7}$\qquad & \qquad (b): $P_{8}$\\
\labellist
\small
\pinlabel {$2$} at 35 75
\pinlabel {$3$} at 71 82
\pinlabel {$3$} at 43 67
\pinlabel {$\dfrac{\pi}{3}$} at 3 50
\pinlabel {$3$} at 35 44
\pinlabel {$2$} at 46 49
\pinlabel {$e_g$} at 46 57
\pinlabel {$\frac{2\pi}{m}$} at 62 59
\pinlabel {$2$} at 76 58
\pinlabel {$\dfrac{\pi}{3}$} at 87 72
\pinlabel {$e_1$} at 25 30
\pinlabel {\turnbox{30}{$\pi/3$}} at 18 15
\pinlabel {\turnbox{15}{$\pi/2$}} at 34 6
\pinlabel {\turnbox{45}{$4\pi/m$}} at 77 29
\endlabellist
\includegraphics[width=3.6 cm]{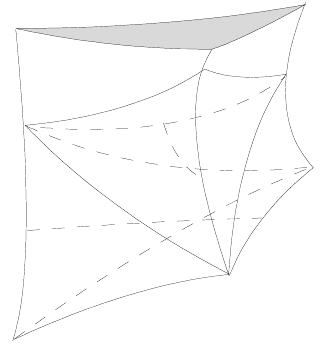} \qquad &
\qquad
\labellist
\small
\pinlabel {$3$} at 28 114
\pinlabel {$2$} at 30 98
\pinlabel {$e_g$} at 46 100
\pinlabel {$2$} at 67 109
\pinlabel {$\dfrac{\pi}{5}$} at 95 96
\pinlabel {$\dfrac{\pi}{5}$} at 3 77
\pinlabel {$3$} at 20 71
\pinlabel {$5$} at 35 76
\pinlabel {$3$} at 65 85
\pinlabel {\turnbox{90}{$\pi/2$}} at 53 58
\pinlabel {$e_1$} at 17 55
\pinlabel {$2$} at 26 42
\pinlabel {\turnbox{20}{$\pi/3$}} at 28 24
\pinlabel {\turnbox{-20}{$\pi/5$}} at 20 14
\pinlabel {\turnbox{65}{$2\pi/5$}} at 82 29
\endlabellist
\includegraphics[width=3.8 cm]{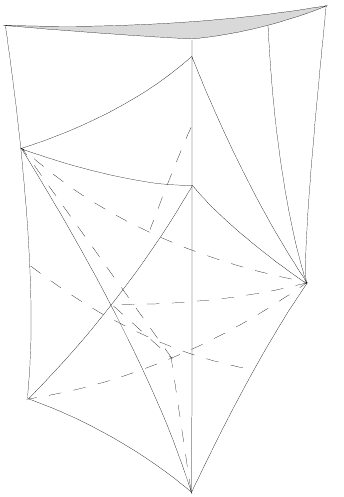}\\
(c): $P_{10}$\qquad & \qquad (d): $P_{16}$\\
\labellist
\small
\pinlabel {$\dfrac{\pi}{5}$} at 65 90
\pinlabel {$A$} at 80 85
\pinlabel {$\dfrac{\pi}{5}$} at 126 100
\pinlabel {$C$} at -5 50 
\pinlabel {$5$} at 15 62
\pinlabel {$\dfrac{2\pi}{5}$} at 55 60
\pinlabel {$5$} at 83 65
\pinlabel {$L$} at 32 40
\pinlabel {$E$} at 48 42
\pinlabel {$M$} at 82 40
\pinlabel {$\dfrac{3\pi}{5}$} at 7 25
\pinlabel {$B$} at 23 20
\pinlabel {$2$} at 42 30
\pinlabel {$V$} at 70 25
\pinlabel {$D$} at 120 25
\endlabellist
\includegraphics[width=5.1 cm]{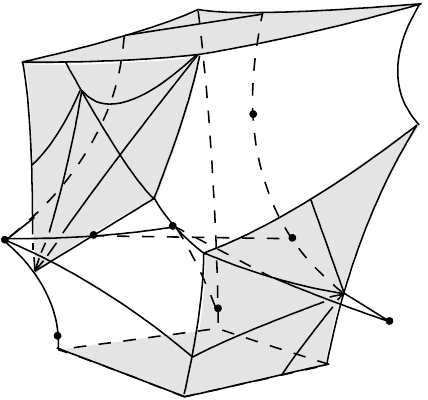}\qquad &
\labellist
\small
\pinlabel {$E$} at 35 137
\pinlabel {$3$} at 69 134
\pinlabel {$A$} at 119 137
\pinlabel {$5$} at 31 117
\pinlabel {$5$} at 75 122
\pinlabel {$M$} at 135 118
\pinlabel {$L$} at 25 104
\pinlabel {$e_g$} at 44 107
\pinlabel {$2$} at 59 110
\pinlabel {$5$} at 143 111
\pinlabel {$D$} at 154 104
\pinlabel {$C$} at 4 73
\pinlabel {$2$} at 50 73
\pinlabel {$3$} at 68 87
\pinlabel {$3$} at 126 70
\pinlabel {$5$} at 40 46
\pinlabel {$3$} at 59 39
\pinlabel {$B$} at 55 5
\pinlabel {$5$} at 82 5
\pinlabel {$V$} at 109 5
\endlabellist
\qquad \includegraphics[width=5.4 cm]{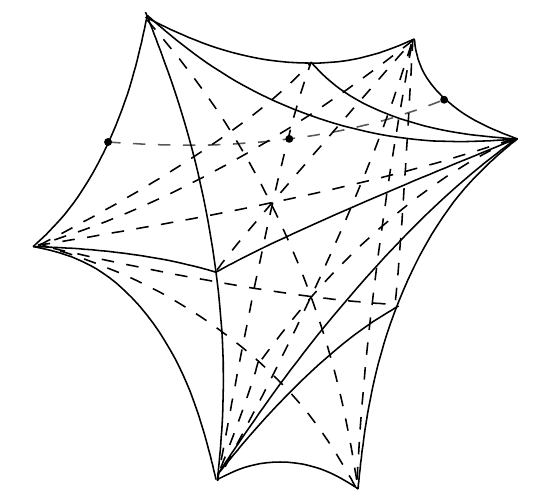}\\
\multicolumn{2}{c}{(e): $P_{19}$}\\
\end{tabular}
\caption{}\label{gamma_noncentral}
\end{figure}

Now consider discrete groups for which $R_\zeta\not\in\Gamma^*$.
In all these cases $\ell=p/3$, where $(p,3)=1$. Let $\eta$ be the
plane
through $\alpha'\cap\zeta$ that makes a dihedral angle of
$2\pi/p$
with~$\alpha'$ and let
$\overline\P=\P(\alpha,\alpha',\omega,\delta,\eta)$.
Denote by $\theta_1$ and $\theta_2$ dihedral angles of
$\overline\P$ at $\eta\cap\alpha$ and $\eta\cap\omega$,
respectively.

If $\P$ is compact or non-compact with $m<\infty$, $\overline\P$
is determined by values $n$, $m$, $\ell$, $\theta_1$ and
$\theta_2$.
For $m=\overline\infty$, we give the distance $d$ between
$\alpha$
and $\alpha'$.

$\mathbf P_7$. \qua
$n\geq 5$, $(n,3)=1$, $m=6$, $\ell=n/3$, $\theta_1=\pi/3$,
$\theta_2=\pi/2$.
$\P(\alpha,\alpha',\omega,\eta)$ is decomposed into tetrahedra
$T=T[2,3,n;2,3,n]$.
A quarter of $T$ is a fundamental polyhedron for $\Gamma^*$
(\fullref{gamma_noncentral}(a)).
$\Gamma^*=\langle G_T,e_1,e_g\rangle$.

$\mathbf P_8$. \qua
$n\geq 5$, $(n,3)=1$, $m=\overline\infty$, $\ell=n/3$,
$\theta_1=\pi/2$, $\theta_2=\pi/n$ and 
$\cosh d=2\cos^2(\pi/n)$.
$\P(\alpha,\alpha',\omega,\eta)$ is decomposed into tetrahedra
$T=T[2,2,4;2,n,4]$.
A half of $T$ is a fundamental polyhedron for $\Gamma^*$
(\fullref{gamma_noncentral}(b)).
$\Gamma^*=\langle G_T,e_1\rangle$.

$\mathbf P_{10}$. \qua
$n=3$, $m\geq 8$ is even, $(m,3)=1$, $\ell=m/6$,
$\theta_1=\pi/2$,
$\theta_2=\pi/3$.
$\P(\alpha,\alpha',\omega,\eta)$ is decomposed into tetrahedra
$T=T[2,3,m/2;2,3,3]$.
A half of $T$ is a fundamental polyhedron for $\Gamma^*$
(\fullref{gamma_noncentral}(c)).
$\Gamma^*=\langle G_T,e_g\rangle=\langle G_T,e_1\rangle$.

$\mathbf P_{16}$. \qua
$n=5$, $m=4$, $\ell=5/3$, $\theta_1=\pi/5$, $\theta_2=2\pi/3$.
$\P(\alpha,\alpha',\omega,\eta)$ is decomposed into tetrahedra
$T=T[2,3,5;2,3,2]$.
A half of $T$ is a fundamental polyhedron for $\Gamma^*$
(\fullref{gamma_noncentral}(d)).
$\Gamma^*=\langle G_T,e_g\rangle=\langle G_T,e_1\rangle$.

$\mathbf P_{19}$. \qua
$n=5$, $m=\overline\infty$, $\ell=5/3$, $\theta_1=3\pi/5$
and $\cosh d=(5+\sqrt{5})/4$. The planes $\eta$ and $\omega$
are disjoint.
$\P(\alpha,\alpha',\omega,\eta)$
is decomposed into tetrahedra $T=T[2,2,3;2,5,3]$.
A half of $T$ is a fundamental polyhedron for $\Gamma^*$, see
\fullref{gamma_noncentral}(e), where $LM=e_g$ and $VE=e_1$.
$\Gamma^*=\langle G_T,e_g\rangle$.

\section{Kleinian orbifolds and their fundamental 
groups}\label{sec:orbifolds}

Let $\Gamma$ be a non-elementary Kleinian group,
and let $\Omega(\Gamma)$ be the discontinuity set of $\Gamma$.
Following Boileau and Porti \cite{BP00}, 
we say that the {\em Kleinian orbifold\/}
$Q(\Gamma)=({\mathbb H}^3\cup\Omega(\Gamma))/\Gamma$
is an orientable $3$--orbifold with a complete hyperbolic
structure
on its interior ${\mathbb H}^3/\Gamma$ and a conformal structure
on its boundary $\Omega(\Gamma)/\Gamma$.

In this section we shall describe the Kleinian orbifold $Q(\Gamma)$
and a presentation for each truly spatial discrete $\mathcal{RP}$
group $(\Gamma;f,g)$ with $f$ elliptic and $g$ hyperbolic.
Since a fundamental polyhedron for $\Gamma^*$
(a finite index extension of $\Gamma$) was shown,
it remains to construct a fundamental polyhedron for
$\Gamma$ itself and identify the equivalent points
on the boundary of the new polyhedron to get the corresponding orbifold.

In figures, we schematically draw singular sets and boundary
components
of the orbifolds using fat vertices and fat edges.
In fact, each picture
gives rise to an infinite series of orbifolds which might be
compact or non-compact of finite or infinite volume.

We say that a finite 3--regular graph $\Sigma(Q)$ 
with fat vertices and fat edges
embedded in a topological space $X$
{\em represents the singular set and/or
boundary components of $Q=Q(\Gamma)$\/} if:
\begin{enumerate}
\item non-fat edges of $\Sigma(Q)$ are labelled by positive
integers greater than 1,
\item fat edges of $\Sigma(Q)$ are labelled by positive 
integers greater than 1
or symbols $\infty$ and $\overline\infty$,
\item the endpoints of a fat edge are fat vertices,
\item if $p$, $q$ and $r$ are labels of the edges incident
to a non-fat vertex, then $1/p+1/q+1/r>1$.
\end{enumerate}
If an edge has no label then the label is meant to be $2$.
To reproduce the orbifold $Q$ from a graph $\Sigma(Q)$
we first work out all fat vertices and then all fat edges
according to labels assigned as follows.

Let $v\in\Sigma(Q)$ be a fat vertex and $p$, $q$ and
$r$ be the labels of the edges incident to $v$.

Suppose that all $p$, $q$, $r<\infty$.
If $1/p+1/q+1/r>1$ then the vertex $v$ is a singular 
point of $Q$ and the local group of $v$ is one of 
the finite groups $D_{2n}$, $S_4$, $A_4$, $A_5$.
If $1/p+1/q+1/r=1$ then $v$ represents a puncture.
A cusp neighborhood of $v$ is a quotient of a horoball in
${\mathbb H}^3$ by a Euclidean triangle group $(2,3,6)$,
$(2,4,4)$ or $(3,3,3)$.
In case $1/p+1/q+1/r<1$ the vertex $v$
must be removed together with its open neighborhood, which
means that $Q$ has a boundary component.

If one of the indices, say $p$, equals $\infty$ and
$1/p+1/q+1/r=1$, then $q=r=2$ and $v$ is a puncture.

For all the other $p$, $q$, $r$, the vertex $v$ is removed
together with its open neighborhood.

Now we proceed with the edges.
If an edge $e$ (fat or non-fat) is labelled by an integer $p<\infty$,
then $e$ is a part of the singular set of the
orbifold $Q$ and consists of cone points of index $p$.

Fat edges labelled by $\infty$ represent cusps of $Q$.
A cusp neighborhood is the quotient of a horoball by an
elementary parabolic group.
Topologically it is $F\times [0,\infty)$, where $F$ is a 
Euclidean orbifold called the cross-section of the cusp
(see eg Boileau--Maillot--Porti \cite{BMP03} 
for geometric structures on orbifolds).

If $e$ is labelled by $\overline\infty$, then it must be 
deleted
together with its open regular neighborhood.
More details on how to `decode' an orbifold with fat edges
and vertices are given in Klimenko--Kopteva \cite{KK05-2}.
We do not discuss them here since fundamental polyhedra for all
$\Gamma$
will be found, so it is not difficult to reconstruct the
orbifolds.

Denote:
\begin{enumerate}
\item
$GT[n,m;q]=\langle f,g|f^n,g^m,[f,g]^q\rangle$,
\item
$PH[n,m,q]=
\langle x,y,z|x^n,y^2,z^2,(xz)^2,[x,y]^m,(yxyz)^q\rangle$,
\item
$H[p;n,m;q]=\langle
x,y,s|s^2,x^n,y^m,(xy^{-1})^p,(sxsy^{-1})^q,(sx^{-1}
y)^2\rangle$,
\item
$P[n,m,q]=\langle
w,x,y,z|w^n,x^2,y^2,z^2,(wx)^2,(wy)^2,(yz)^2,
(zx)^q,(zw)^m\rangle$,
\item $\Tet[p_1,p_2,p_3;q_1,q_2,q_3] = \\ \qua
\langle x,y,z|
x^{p_1},y^{p_2},z^{p_3},(yz^{-1})^{q_1},(zx^{-1})^{q_2},
(xy^{-1})^{q_3}\rangle$,
where, for simplicity, the group $\Tet[2,2,n;2,q,m]$ is
denoted by $\Tet[n,m;q]$,
\item
$GTet_1[n,m,q]=\langle
x,y,z|x^n,y^2,(xy)^m,[y,z]^q,[x,z]\rangle$,
\item
$GTet_2[n,m,q]=\langle
  x,y,z|x^n,y^2,(xy)^m,(xz^{-1}y^{-1}zy)^q,[x,z]\rangle$,
\item
$\mathcal{S}_2[n,m,q]=\langle
x,L|x^n,(xLxL^{-1})^m,(xL^2x^{-1}L^{-2})^q\rangle$,
\item
$\mathcal{S}_3[n,m,q]=\langle
x,L|x^n,(xLxL^{-1})^m,(xLxLxL^{-2})^q\rangle$,
\item
$R[n,m;q]=\langle u,v|(uv)^n,(uv^{-1})^m,[u,v]^q\rangle$.
\end{enumerate}
In the presentations (1)--(10), the exponents $n,m,q,\dots$ may be
integers (greater than 1), $\infty$ or~$\overline\infty$.
We employ the symbols
$\infty$ and $\overline\infty$ in the following way.
If we have relations of the form $w^n=1$, where
$n=\overline\infty$,
we remove them
from the presentation (in fact, this means that the
element $w$ is hyperbolic in the Kleinian group). Further, if we
keep
the relations $w^\infty=1$, we get a Kleinian group
presentation where parabolics are indicated.
To get an abstract group presentation, we need to remove such
relations as well.

The reader can find the orbifolds that correspond to the above
presentations in Figures \ref{gt_tet}, \ref{s2s1} and \ref{rp3}.

\begin{figure}[ht!]
\centering
\begin{tabular}{ccc}
\labellist
\small
\pinlabel {$n$} at 4 73
\pinlabel {$q$} at 82 3
\pinlabel {$m$} at 131 90
\endlabellist
\includegraphics[width=3 cm]{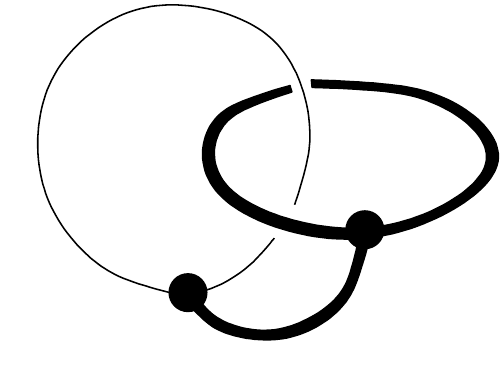}\quad &
\quad
\labellist
\small
\pinlabel {$n$} at 16 48
\pinlabel {$q$} at 72 14
\pinlabel {$m$} at 45 22
\endlabellist
\includegraphics[width=2.3 cm]{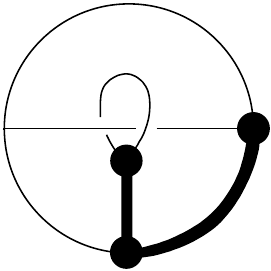} \quad&
\quad 
\labellist
\small
\pinlabel {$n$} at 30 44
\pinlabel {$q$} at 61 2
\pinlabel {$m$} at 83 45
\pinlabel {$p$} at 32 90
\endlabellist
\includegraphics[width=2.5 cm]{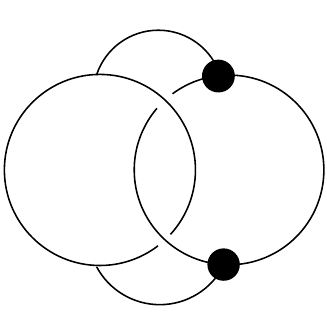}\\
(a): $GT[n,m;q]$\quad &
\quad(b): $PH[n,m,q]$ \quad &
\quad (c): $H[p;n,m;q]$\\
\quad 
\labellist
\small
\pinlabel {$p_1$} at 14 52
\pinlabel {$p_2$} at 38 52
\pinlabel {$p_3$} at 38 20
\pinlabel {$q_1$} at 61 36
\pinlabel {$q_2$} at 44 5
\pinlabel {$q_3$} at 77 54
\endlabellist
\includegraphics[width=2.3 cm]{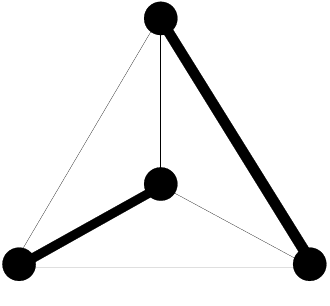}&
\labellist
\small
\pinlabel {$n$} at 36 18
\pinlabel {$q$} at 44 3
\pinlabel {$m$} at 76 53
\endlabellist
\includegraphics[width=2.3 cm]{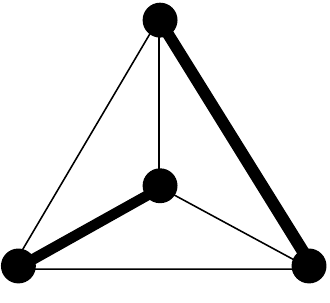} &
\labellist
\small
\pinlabel {$n$} at 26 11
\pinlabel {$q$} at 78 46
\pinlabel {$m$} at 64 11
\endlabellist
\includegraphics[width=2.3 cm]{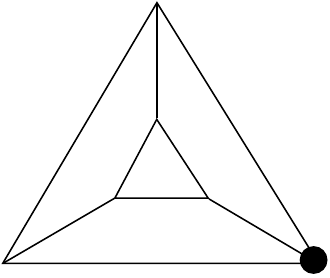}\\
(d): $\Tet[p_1,p_2,p_3;q_1,q_2,q_3]$&
(d'): $\Tet[n,m;q]$&
(e): $P[n,m,q]$\\
\end{tabular}
\caption{Orbifolds embedded in ${\mathbb S}^3$}\label{gt_tet}
\end{figure}
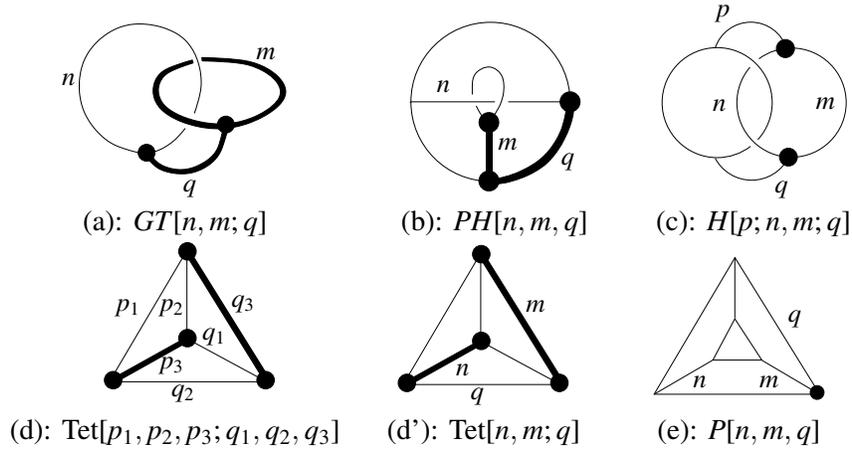

We start with description of presentations and orbifolds for all
truly spatial discrete groups generated by a primitive
elliptic and a hyperbolic elements
with {\em disjoint axes} (\fullref{groups_disjoint}).
All such orbifolds are embedded in $\mathbb S^3$.

As
usual, we can also apply the theorem when the elliptic generator is
non-primitive, using recalculation formulas for parameters
as follows.
Suppose that $f$ is a non-primitive elliptic of finite order $n$,
that is, $\beta(f)=-4\sin^2(q\pi/n)$, where $(q,n)=1$ and $1<q<n/2$.
Then there exists an integer $r$ so that $f^r$
is primitive of the same order.
Obviously, $\langle f,g\rangle=\langle f^r,g\rangle$
and $\beta(f^r)=-4\sin^2(\pi/n)$. 
By Gehring and Martin \cite{GM94-1},
$\gamma(f,g)=\gamma(f^r,g)\cdot\beta(f)/\beta(f^r)$.

\begin{theorem}\label{groups_disjoint}
Let $(\Gamma;f,g)$ be a discrete $\mathcal{RP}$ group
with
$\beta=-4\sin^2(\pi/n)$, $n\geq 3$,
$\beta'\in(0,+\infty)$ and
$\gamma\in(-\infty,0)$. Then one of the following occurs.
\begin{enumerate}
\item $\gamma=-4\cosh^2u$, where $u\in \mathcal{U}$, $(t(u),2)=2$
and $t(u)\geq 4$;
$\Gamma$ is isomorphic to
$GT[n,\overline\infty;t(u)/2]$.
\item $\gamma=-4\cosh^2u$, where $u\in \mathcal{U}$, $(t(u),2)=1$
and $t(u)\geq 3$;
$\Gamma$ is isomorphic to
$\Tet[n,\overline\infty;t(u)]$.
\item $n\geq 5$, $(n,2)=1$, $\gamma=-(\beta+2)^2$ and
$\beta'=4(\beta+4)\cosh^2u-4$, where $u\in \mathcal{U}$ and
$t(u)\geq 4$;
$\Gamma$ is isomorphic to
$\Tet[n,t(u);3]$.
\item $\beta=-3$, $\gamma=(\sqrt{5}-3)/2$ and
$\beta'=2(7+3\sqrt{5})\cosh^2u-4$, where $u\in \mathcal{U}$ and
$t(u)\geq 3$;
$\Gamma$ is isomorphic to
$\Tet[3,t(u);5]$.
\end{enumerate}
\end{theorem}

\begin{proof}
All parameters for discrete groups in the statement of
\fullref{groups_disjoint} are described in
\fullref{param_disjoint}.
We shall obtain
a presentation for each discrete group
by using the Poincar\'e polyhedron theorem, 
see eg Epstein and Petronio \cite{EP94}.

Let $\Gamma$ have parameters as in part (1) of
\fullref{param_disjoint}.
In \fullref{construction_disjoint}, a fundamental polyhedron
for the group $\Gamma^*$ was described.
Since $\widetilde\Gamma$ is the orientation preserving index 2
subgroup of $\Gamma^*$, we can take
$\widetilde\P=\P(\varepsilon,\alpha,\tau,R_\omega(\alpha))$ as
a fundamental polyhedron for $\widetilde\Gamma$
(see \fullref{fp_disjoint_1}(a)).
In our notation $p=t(u)$.

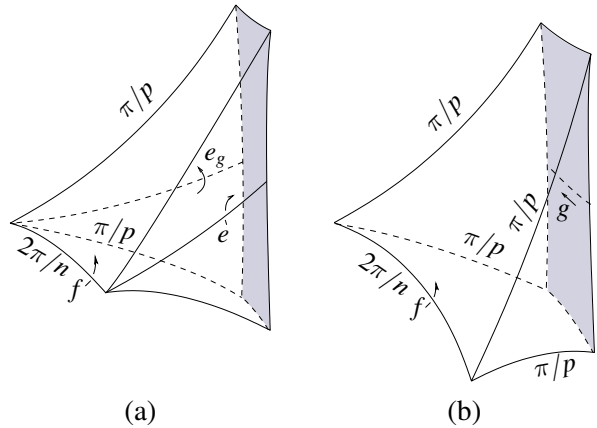
\begin{figure}[ht!]
\centering
\begin{tabular}{cc}
\labellist
\small
\pinlabel {$f$} at 45 60
\pinlabel {$e_g$} at 136 148
\pinlabel {$e$} at 141 96
\pinlabel {\turnbox{45}{$\pi/p$}} at 77 173
\pinlabel {\turnbox{-10}{$\pi/p$}} at 55 98
\pinlabel {\turnbox{-30}{$2\pi/n$}} at 5 87
\endlabellist
\includegraphics[width=3.5 cm]{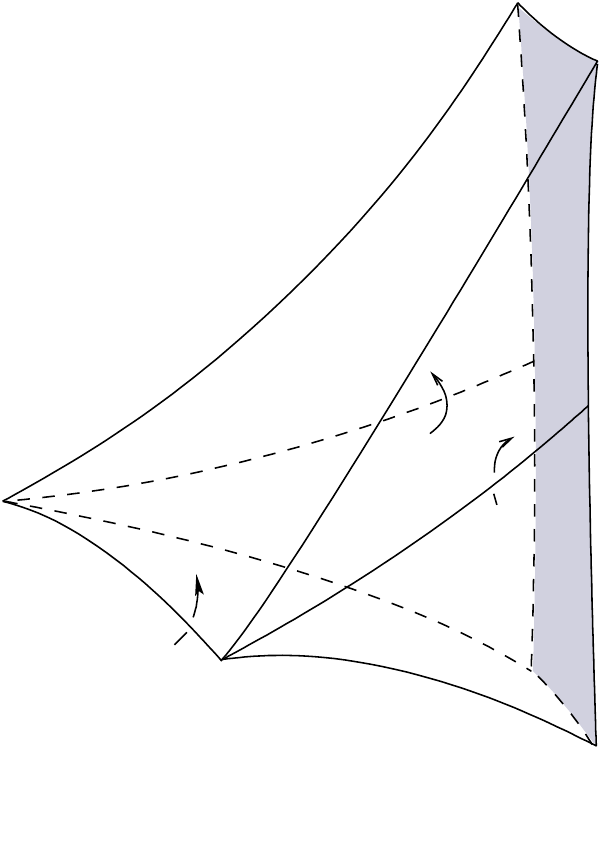}\quad &
\quad 
\labellist
\small
\pinlabel {$f$} at 53 45
\pinlabel {$g$} at 141 104
\pinlabel {\turnbox{45}{$\pi/p$}} at 62 153
\pinlabel {\turnbox{-20}{$\pi/p$}} at 79 86
\pinlabel {\turnbox{60}{$\pi/p$}} at 113 91
\pinlabel {\turnbox{-35}{$2\pi/n$}} at 15 72
\pinlabel {\turnbox{10}{$\pi/p$}} at 124 3
\endlabellist
\includegraphics[width=3.5 cm]{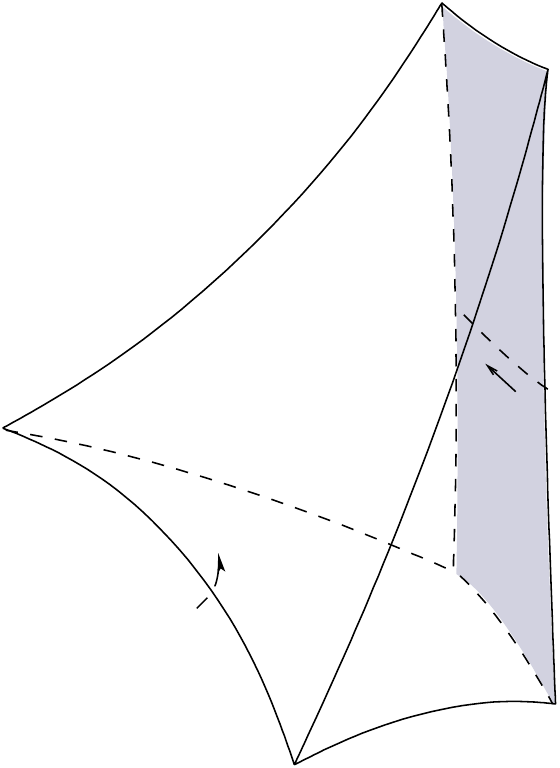}\\
(a)\quad & \quad (b)
\end{tabular}
\caption{Fundamental polyhedra for 
$\widetilde \Gamma$ and $\Gamma$ in case of disjoint axes}
\label{fp_disjoint_1}
\end{figure}

Let $e_g=R_\tau R_\omega$. It is clear that $e_g=ge$.
By applying the Poincar\'e polyhedron theorem to
$\widetilde{\mathcal{P}}$ and face pairing
transformations $e$, $e_g$ and $f$, we get
$$
\widetilde\Gamma=\langle e,e_g,f|e^2,e_g^2,f^n,(fe)^2,
(fe_g)^p\rangle,
$$
where $p$ is an integer, $\infty$ or $\overline\infty$.
Since $g=e_ge$,
$$
\widetilde\Gamma=\langle f,g,e|f^n,e^2,(fe)^2,(ge)^2,
(fge)^p\rangle.
$$
If $p$ is odd, then from the relations for $\widetilde\Gamma$ it
follows that
$e=(fgf^{-1}g^{-1})^{(p-1)/2}fg$. Hence, in this case
$\widetilde\Gamma=\Gamma$ and $\Gamma\cong
\Tet[n,\overline\infty;p]$. The isomorphism is given by $f\mapsto z$,
$g\mapsto xy^{-1}$ and $e\mapsto y$.
Identifying faces of $\widetilde{\mathcal{P}}$, we get the orbifold
$Q(\Gamma)$ shown in \fullref{gt_tet}(d').

If $p$ is even, $\infty$ or $\overline\infty$,
then $\Gamma$ is a subgroup of index~2 in $\widetilde\Gamma$.
To see this, we apply the Poincar\'e theorem to the polyhedron
$\P(\alpha,\tau,R_\omega(\alpha),R_\varepsilon(\tau))$
(see \fullref{fp_disjoint_1}(b)). Then
$$
\Gamma=\langle f,g|f^n,(fgf^{-1}g^{-1})^{p/2}\rangle
\cong GT[n,\overline\infty;p/2].
$$
The orbifold $Q(\Gamma)$ is shown in
\fullref{gt_tet}(a).

Now consider the groups with parameters from part~(2) of
\fullref{param_disjoint}.
In this case $t(u)=q$ from \fullref{discr_disjoint}.
By applying the Poincar\'e theorem to the polyhedron
$\P(\varepsilon,\alpha,\xi_1,R_\omega(\alpha),R_\omega(\xi_1))$
and the group generated by $f$, $e$ and $s$, where
$s=R_\omega R_{\xi_1}$, we get the following presentation
for the group $\langle f,e,s\rangle$:
$$
\langle f,e,s|f^n,e^2,s^3,(fe)^2,(fs)^2,(se)^q\rangle.
$$
Since $x=R_\kappa R_\tau$, we have $x^2=h$ and $x=fs^{-1}$. Therefore,
$g=e_ge=f^{-1}he=f^{-1}x^2e=f^{-1}(fs^{-1})^2e=s^{-1}fs^{-1}e$
and hence $\Gamma\subseteq\langle f,e,s\rangle$.

Since $h^n=1$, $n$ is odd and $h^2=[f,g]$,
we have that $h=[f,g]^{-(n-1)/2}\in\Gamma$. Further,
$e_g=f^{-1}h$ and so $e=e_gg=f^{-1}hg\in\Gamma$. From $x^n=1$ we have
that $x=h^{-(n-1)/2}\in\Gamma$ and, therefore,
$s=x^{-1}f\in\Gamma$.
Then $\langle f,e,s\rangle\subseteq\Gamma$ and so we have shown that
$\Gamma=\langle f,e,s\rangle$.

Mapping $x\mapsto s^{-1}f^{-1}$, $y\mapsto fe$, $z\mapsto f$,
we see that $\langle f,e,s\rangle=\Gamma$ is isomorphic to the group
$\Tet[n,q;3]$. Therefore, $\Gamma\cong \Tet[n,q;3]$, where $q\geq 4$
is an integer, $\infty$ or $\overline\infty$.

The orbifold $Q(\Gamma)$ is shown in \fullref{gt_tet}(d').

Similarly, one can show that the groups with parameters
from part (3) of \fullref{param_disjoint} are
isomorphic to $\Tet[3,t(u);5]$, where $t(u)\geq 3$
is an integer, $\infty$ or $\overline\infty$.
\end{proof}

Let $T(p)$, $ p\in{\mathbb Z}$, be a Seifert fibered solid torus
obtained from a trivially
fibered solid torus $D^2\times {\mathbb S}^1$ by cutting it along
$D^2\times \{x\}$ for some $x\in {\mathbb S}^1$, rotating one of the
discs through $2\pi/p$ and gluing back together.

Denote by
$\mathcal{S}(p)$ a space obtained by gluing $T(p)$ with its mirror
symmetric copy along their boundaries fiber to fiber.
Clearly, $\mathcal{S}(p)$ is homeomorphic to 
${\mathbb S}^2\times {\mathbb S}^1$
and is $p$-fold covered by trivially fibered
${\mathbb S}^2\times {\mathbb S}^1$.
There are two critical fibers\footnote{A critical fiber is
also called a {\em singular\/} fiber. We use the word `critical'
in order not to confuse it with components of the singular
sets of orbifolds.}
in $\mathcal{S}(p)$ whose `length' is
$p$ times shorter
than the `length' of a regular fiber.

Next two theorems describe presentations and orbifolds for
all truly spatial discrete groups
$\Gamma=\langle f,g\rangle$ whose generators have {\em
intersecting axes\/},
$g$ is hyperbolic and $f$ is primitive elliptic of even order
(\fullref{groups_even})
or odd order (\fullref{groups_odd}).
In both theorems there are series of orbifolds embedded into
$\mathbb S^3$ and
${\mathbb S}^2\times {\mathbb S}^1$; in case when $f$ has odd order some
orbifolds
are embedded into $\mathbb{R}P^3$.

Again for non-primitive elliptics we can use recalculation formulas
for parameters to apply \fullref{groups_even} 
or \fullref{groups_odd}
(see the paragraph before \fullref{groups_disjoint}).

\begin{theorem}\label{groups_even}
Let $(\Gamma;f,g)$ be a discrete $\mathcal{RP}$ group
so that $\beta=-4\sin^2(\pi/n)$, $n\geq 4$, $(n,2)=2$,
$\beta'\in(0,+\infty)$ and
$\gamma\in(0,-\beta\beta'/4)$.
Then $\gamma=4\cosh^2u+\beta$, where $u\in \mathcal{U}$ and
$1/n+1/t(u)<1/2$,
and one of the following occurs.
\begin{enumerate}
\item $(t(u),2)=2$  and $\beta'=4(\cosh^2v)/\gamma-4\gamma/\beta$,
where $v\in \mathcal{U}$, $t(v)\geq 3$ and $(t(v),2)=1$;
$\Gamma$ is isomorphic to $PH[n,t(u)/2,t(v)]$.
\item $(t(u),2)=2$ and $\beta'=4(\cosh^2v)/\gamma-4\gamma/\beta$,
where $v\in \mathcal{U}$, $t(v)\geq 4$ and $(t(v),2)=2$;
$\Gamma$ is isomorphic to $\mathcal{S}_2[n,t(u)/2,t(v)/2]$.
\item $(t(u),2)=1$ and
$\beta'=4(\gamma-\beta)(\cosh^2v)/\gamma-4\gamma/\beta$,
where $v\in \mathcal{U}$, $t(v)\geq 3$ and $(t(v),2)=1$;
$\Gamma$ is isomorphic to $P[n,t(u),t(v)]$.
\item $(t(u),2)=1$ and
$\beta'=4(\gamma-\beta)(\cosh^2v)/\gamma-4\gamma/\beta$,
where $v\in \mathcal{U}$, $t(v)\geq 4$ and $(t(v),2)=2$;
$\Gamma$ is isomorphic to $GTet_1[n,t(u),t(v)/2]$.
\item $\beta=-2$, $(t(u),2)=1$, $t(u)\geq 5$ and
$\beta'=\gamma^2+4\gamma$;
$\Gamma$ is isomorphic to $\Tet[4,t(u);3]$.
\end{enumerate}
\end{theorem}

\begin{proof}
The idea of the proof is the same as for \fullref{groups_disjoint}.
We refer now to the part of \fullref{construction_intersect}
where $n$ is even.

{\bf 1.} \qua
Let $\Gamma$ have parameters as in row $P_1$ of
\fullref{table_param}.
A fundamental polyhedron
$\P(\alpha,\alpha',\delta,\varepsilon,\omega)$
for $\Gamma^*$ is shown in \fullref{fp_even}(a).
A fundamental polyhedron for $\widetilde\Gamma$ is
$\P(\alpha,\alpha',R_\omega(\alpha),R_\omega(\alpha'),\delta,
\varepsilon)$,
whose faces are identified by
face pairing transformations $f$, $f'=R_\omega R_{\alpha'}$,
$e_2=f^{n/2}e$ and $e_g$.
(We doubled the
fundamental polyhedron for $\Gamma^*$ shown in \fullref{fp_even}(a).)
By the Poincar\'e polyhedron theorem, we get that
$$\widetilde\Gamma=
\langle f,f',e_g,e_2|f^n,( f')^n,e_g^2,e_2^2,(fe_2)^2,
e_gf^{-1}e_gf',( f^{-1}f')^{m/2},(e_2f')^\ell\rangle.
$$

Since $e_g=ge$ and $e_2=f^{n/2}e$, we have that
$$\widetilde\Gamma=
\langle f,g,e|f^n,e^2,(fe)^2,(ge)^2,(gfg^{-1}f)^{m/2},
(f^{n/2}g^{-1}fge)^\ell\rangle.
$$
If $\ell$ is odd, $e\in\langle f,g\rangle$.
Therefore, in this case
$\Gamma=\widetilde\Gamma\cong PH[n,m/2,\ell]$, where $m/2$ is
an integer, $\infty$ or $\overline\infty$ and $\ell$ is odd;
the orbifold $Q(\Gamma)$ is shown in \fullref{gt_tet}(b).

Suppose now that $\ell$ is even.
Consider the polyhedron
$\P'$ bounded by $\alpha$, $\alpha'$, $\varepsilon$,
$R_\omega(\alpha)$, $R_\omega(\alpha')$ and $e_g(\varepsilon)$.
The $\pi$--loxodromic element $L=e_ge_2$ identifies the faces
of $\P'$ lying in $\varepsilon$ and $e_g(\varepsilon)$.
Applying the Poincar\'e polyhedron theorem to $\P'$ and
the transformations $f$, $f'$ and $L$, we get that
$\langle f,f', L\rangle$ is discrete and $\P'$ is a fundamental polyhedron
for it.
It follows, in particular, that $|\widetilde\Gamma:\Gamma|=2$ for
even $\ell$.
Moreover, $\langle f,f', L\rangle$ has
the following presentation:
$$
\langle f,f',L|f^n,(f')^n,(f^{-1}f')^{m/2},L^{-1}f'Lf,
(L^{-1}fLf')^{\ell/2}\rangle.
$$
Since $f'=Lf^{-1}L^{-1}$, the group
$\langle f,f',L\rangle$ is generated by $f$ and $L$ and
is isomorphic to ${\mathcal S}_2[n,m/2,\ell/2]$.
Further, since $L=e_ge_2=gf^{n/2}$, the group
$\langle f,L\rangle$ coincides with $\Gamma$.
Therefore, $\Gamma\cong{\mathcal S}_2[n,m/2,\ell/2]$,
where $m/2$ and $\ell/2$ are
integers, $\infty$ or $\overline\infty$;
the orbifold $Q(\Gamma)$ is shown in \fullref{s2s1}(a),
see also \fullref{rem_fibre} and \fullref{rem_cover} after the proof.

{\bf 2.} \qua
Now let $\Gamma$ have parameters as in row $P_2$ of
 \fullref{table_param}.
A fundamental polyhedron for $\widetilde\Gamma$ is
$\P(\alpha,\delta,\varepsilon,\xi,R_\omega(\alpha))$, whose faces
are
identified by $f$, $e_2$,
$y=R_\delta R_\omega$ and $z=R_\omega R_\xi$.  Then
$$
\widetilde\Gamma=
\langle f,e_2,y,z|f^n,e_2^2,y^2,z^2,(yz)^2,(yf)^2,(fe_2)^2,
(ze_2)^k,(fz)^m\rangle.
$$
Using the facts that $e_g=yz=ge$ and $yfy=f^{-1}$, we get
$zfz=(zy)(yfy)(yz)=gef^{-1}ge=gfg^{-1}$.
Therefore, since $m$ is odd, $z=z(f,g)$.
Furthermore, since $e_2=f^{n/2}e$,
$\langle f,e_2,y,z\rangle=\widetilde\Gamma$.
Similarly to part 1 above, if $k$ is odd then
$\widetilde\Gamma=\Gamma$ since in this case $e=e(f,g)\in\Gamma$.
Further, the group $\langle f,e_2,y,z\rangle$ is obviously isomorphic
to $P[n,m,k]$, where $m<\infty$ is also odd.
The orbifold $Q(\Gamma)$ is
shown in \fullref{gt_tet}(e).

If $k$ is even, $\Gamma$ is an index 2 subgroup in
$\widetilde\Gamma$.
The polyhedron
$\P(\alpha,\varepsilon,\xi,R_\omega(\alpha),R_\delta(\varepsilon))$,
whose faces are identified by $f$, $z$ and
$u=ye_2=zgf^{n/2}\in\Gamma$, satisfies the hypotheses of the Poincar\'e
polyhedron theorem. Then $\langle f,z,u\rangle$ is discrete and
has presentation
$$
\langle f,z,u|f^n,z^2,(zf)^m,[z,u]^{k/2},[f,u]\rangle
$$
and $\P(\alpha,\varepsilon,\xi,R_\omega(\alpha),R_\delta(\varepsilon))$
is a fundamental polyhedron for this group.

Obviously, $\langle f,z,u\rangle$ is isomorphic to $GTet_1[n,m,k/2]$.
On the other hand, since $u=zgf^{n/2}$, we have
$\langle f,z,u\rangle=\langle f,g,z\rangle$. Moreover, 
$z=z(f,g)$ because $m$ is odd. Hence,
$\langle f,g,z\rangle= \Gamma$ and, therefore,
$\Gamma$ is isomorphic to $GTet_1[n,m,k/2]$, where
$m<\infty$ is odd and
$k/2$ is an integer, $\infty$ or $\overline\infty$.
The orbifold $Q(\Gamma)$ is shown in \fullref{s2s1}(d).

{\bf 3.} \qua
If $\Gamma$ has parameters as in row $P_3$ of
\fullref{table_param},
it is easy to show that $\Gamma=\widetilde\Gamma$ and $\Gamma$ is
isomorphic
to a tetrahedron group $\Tet[4,m;3]$, where $5\leq m<\infty$ is
odd.
\end{proof}

\begin{figure}[ht!]
\centering
\begin{tabular}{cc}
\labellist
\small
\pinlabel {$q$} at 98 22
\pinlabel {$m$} at 125 26
\pinlabel {$n$} at 132 16
\endlabellist
\includegraphics[width=6 cm]{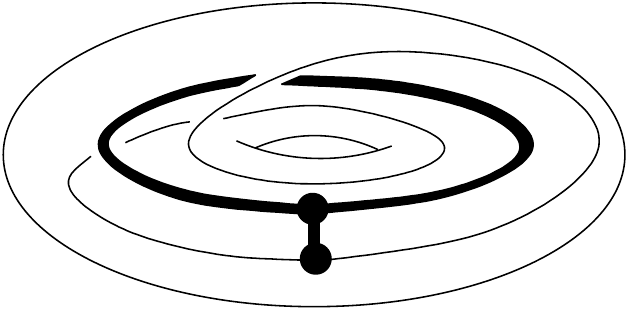}&
\labellist
\small
\pinlabel {$q$} at 100 31
\pinlabel {$m$} at 56 37
\pinlabel {$n$} at 80 6
\pinlabel {$2$} at 65 20
\endlabellist
\includegraphics[width=6 cm]{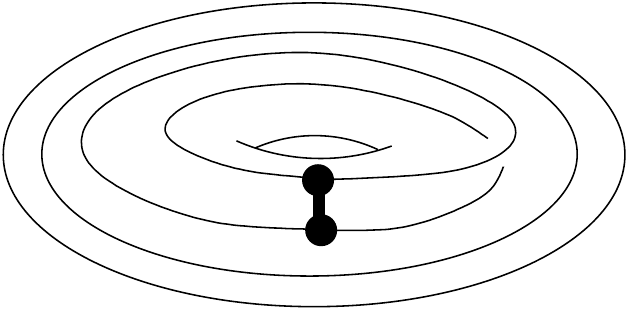}\\
(a) Orbifolds embedded in $\mathcal{S}(2)$; &
(b) Orbifolds embedded in $\mathcal{S}(2)$;\\
$\pi_1^{\orb}(Q)\cong \mathcal{S}_2[n,m,q]$ &
$\pi_1^{\orb}(Q)\cong GTet_2[n,m,q]$\\
\labellist
\small
\pinlabel {$q$} at 100 32
\pinlabel {$m$} at 57 38
\pinlabel {$n$} at 70 19
\endlabellist
\includegraphics[width=6 cm]{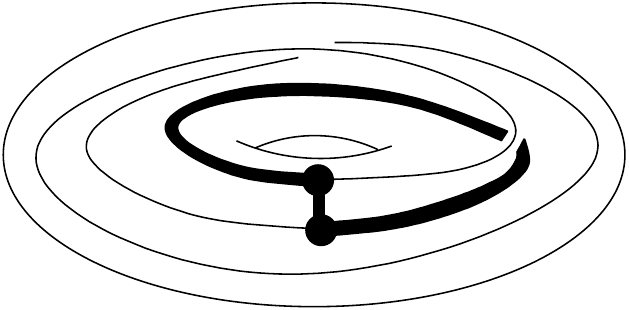}&
\labellist
\small
\pinlabel {$q$} at 98 22
\pinlabel {$m$} at 156 23
\pinlabel {$n$} at 131 49 
\pinlabel {$2$} at 146 35
\endlabellist
\includegraphics[width=6 cm]{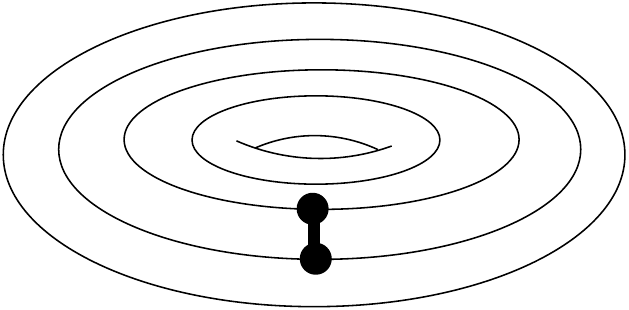}
\\
(c) Orbifolds embedded in $\mathcal{S}(3)$; &
(d) Orbifolds embedded in ${\mathbb S}^2\times {\mathbb S}^1$;\\
$\pi_1^{\orb}(Q)\cong \mathcal{S}_3[n,m,q]$ &
$\pi_1^{\orb}(Q)\cong GTet_1[n,m,q]$
\end{tabular}
\caption{Orbifolds embedded in Seifert fiber spaces;
only the torus that contains cone points or boundary components
is shown.}\label{s2s1}
\end{figure}
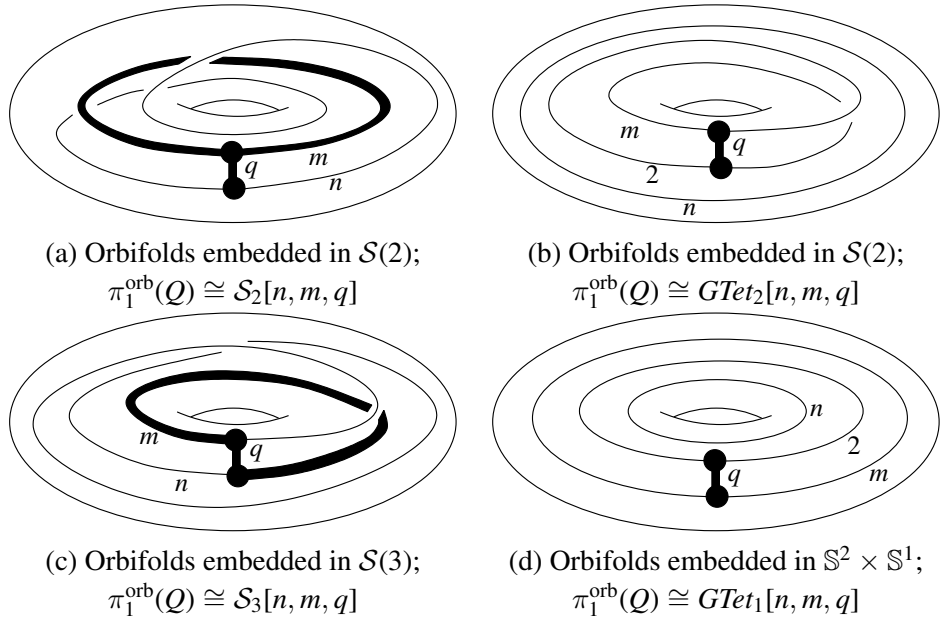

\begin{rem}\label{rem_fibre}
Note that when $Q=Q(\mathcal{S}_2[n,m/2,\ell/2])$, due to the action
of the face pairing transformation of the fundamental polyhedron,
$Q$ is embedded in a Seifert fiber space $\mathcal{S}(2)$
and the singular set is placed in $\mathcal{S}(2)$ in such
a way that the axis of order $m$ (if $m<\infty$)
lies on a critical fiber
of $\mathcal{S}(2)$ and the axis of order $n$ lies on a regular one.
In \fullref{s2s1}(a) we draw only the solid torus that
contains
singular points (or boundary components). The other fibered torus
is
meant to be attached and is not shown.
\end{rem}

\begin{rem}\label{rem_cover}
As an illustration of the orbifold covering 
$Q(\Gamma)\to Q(\widetilde \Gamma)$,
consider the case when parameters of $\Gamma$ are as in row
$P_1$ of \fullref{table_param} 
and $t(u)=\ell$ is even.
Denote
$Q=Q(\Gamma)$ and $\widetilde Q=Q(\widetilde \Gamma)$,
where $\Gamma\cong \mathcal{S}_2[n,m/2,\ell/2]$ and
$\widetilde\Gamma\cong PH[n,m/2,\ell]$.
Let us show the structure
of the orbifold covering $\pi\co Q\to\widetilde Q$.
Assume for simplicity that $m,\ell<\infty$.
Draw the orbifold $Q$ (same as in \fullref{s2s1}(a),
but with the change of indices $q\mapsto \ell/2$,
$m\mapsto m/2$) in the spherical shell
${\mathbb S}^2\times I$ as shown in \fullref{covering};
keep in mind that the inner and outer spheres
are identified. In \fullref{covering}, the labels on
the upper left and the lower right pictures are integers
and denote the cone singularities. The labels on the central
pictures (which show the structure of the covering) are of the
form $2\pi/k$; they indicate cone/dihedral angles.

Let $\sigma$ be a circle
in the $xy$-plane such that the inversion in the sphere
for which $\sigma$ is a big circle identifies the inner and the
outer spheres.
Let $s$ be the orientation preserving automorphism of $Q$ induced
by the composition of this inversion and the reflection in the
$xy$--plane. Thus,
$s$ is of order $2$ with the axis $\sigma$.
Then $s$ determines $\pi\co Q\to \widetilde Q$
and $\langle \pi_1^{\orb}(Q),s\rangle=\pi_1^{\orb}(\widetilde Q)$.
The underlying space of $\widetilde Q$ is ${\mathbb S}^3$.
\end{rem}

\begin{figure}[ht!]
\centering
\labellist
\small
\pinlabel {$n$} at 11 102
\pinlabel {\footnotesize{$m/2$}} at 22 98
\pinlabel {\footnotesize{$l/2$}} at 36 92
\pinlabel {$n$} at 65 110
\pinlabel {$=$} at 87 105
\pinlabel {$z$} at 140 151
\pinlabel {\scriptsize{$2\pi/n$}} at 144 132
\pinlabel {$\sigma$} at 153 113
\pinlabel {\scriptsize{$2\pi/l$}} at 135 105
\pinlabel {\scriptsize{$4\pi/m$}} at 117 96
\pinlabel {\scriptsize{$2\pi/l$}} at 148 86
\pinlabel {\scriptsize{$2\pi/n$}} at 124 79
\pinlabel {$x$} at 102 76
\pinlabel {$2$} at 140 62
\pinlabel {$y$} at 182 109
\pinlabel {$Q=Q(\Gamma)$} at 190 138
\pinlabel {$\pi/2$} at 157 53
\pinlabel {\scriptsize{$\pi/2$}} at 133 39
\pinlabel {$s$} at 114 42
\pinlabel {\scriptsize{$\pi/l$}} at 135 33
\pinlabel {$\sigma$} at 159 33
\pinlabel {\scriptsize{$2\pi/m$}} at 118 24 
\pinlabel {\turnbox{-40}{\scriptsize{$\pi/2$}}} at 100 32
\pinlabel {\scriptsize{$\pi/l$}} at 153 16
\pinlabel {\scriptsize{$2\pi/n$}} at 143 7
\pinlabel {$=$} at 185 34
\pinlabel {$\widetilde{Q}=Q(\widetilde{\Gamma})$} at 228 67
\pinlabel {$m/2$} at 216 23
\pinlabel {$n$} at 213 7
\pinlabel {$l$} at 256 13
\endlabellist
\includegraphics[width=12 cm]{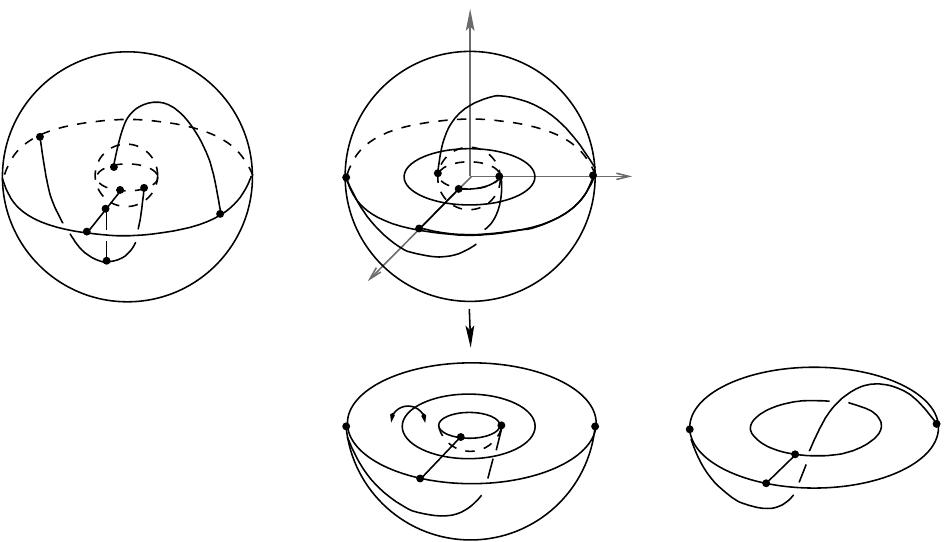}
\caption{Branched covering 
$\pi\co Q\to\widetilde Q$}\label{covering}
\end{figure}

\begin{theorem}\label{groups_odd}
Let $(\Gamma;f,g)$ be a discrete $\mathcal{RP}$ group
so that
$\beta=-4\sin^2(\pi/n)$, $n\geq 3$, $(n,2)=1$,
$\beta'\in(0,+\infty)$ and $\gamma\in(0,-\beta\beta'/4)$.
Then one of the following occurs.
\begin{enumerate}
\item
$\gamma=4\cosh^2u+\beta$,
where $u\in \mathcal{U}$, $(t(u),2)=2$, $1/n+1/t(u)<1/2$,
and
$\beta'=\frac2{\gamma}\left(\cosh v-\cos(\pi/n)\right)-
\frac2{\gamma\beta}
\left((\gamma-\beta)^2\cos(\pi/n)+\gamma(\gamma+\beta)\right)$,
where $v\in \mathcal{U}$;
$\Gamma$ is isomorphic to $\mathcal{S}_3[n,t(u)/2,t(v)]$.
\item
$\gamma=4\cosh^2u+\beta$,
where $u\in \mathcal{U}$, $(t(u),2)=1$, $1/n+1/t(u)<1/2$,
and
$\beta'=\frac{2(\gamma-\beta)}\gamma\cosh v-
\frac2{\gamma\beta}
\left((\gamma-\beta)^2\cos(\pi/n)+\gamma(\gamma+\beta)\right)$,
where $v\in \mathcal{U}$;
$\Gamma$ is isomorphic to $GTet_2[n,t(u),t(v)]$.
\item
$n\geq 7$, $\gamma=(\beta+4)(\beta+1)$ and
$\beta'=2(\beta+2)^2(\cosh v-\cos(\pi/n))/(\beta+1)-
2\left(\beta^2+6\beta+4\right)/\beta$,
where $v\in \mathcal{U}$;
$\Gamma$ is isomorphic $GTet_2[n,3,t(v)]$.
\item
$\beta=-3$, $\gamma=2\cos(2\pi/m)-1$, where $m\geq 7$, $(m,2)=1$,
and
$\beta'=2(\gamma^2+2\gamma+2)/\gamma$;
$\Gamma$ is isomorphic to $GTet_1[m,3,2]$.
\item
$n\geq 5$, $(n,3)=1$, $\gamma=\beta+3$ and
$\beta'=2\left((\beta-3)\cos(\pi/n)-2\beta-3\right)/\beta$;
$\Gamma$ is isomorphic to $H[2;3,n;2]$.
\item
$(\beta,\gamma,\beta')=
((\sqrt{5} -5)/2,(\sqrt{5}\pm 1)/2,3(\sqrt{5}+1)/2)$;
$\Gamma$ is isomorphic to $H[2;2,5;3]$.
\item
$(\beta,\gamma,\beta')=(-3,(\sqrt{5}\pm 1)/2,\sqrt{5})$ or
$(\beta,\gamma,\beta')=((\sqrt{5}-5)/2,(\sqrt{5}-1)/2,\sqrt{5})$,
or
$(\beta,\gamma,\beta')=((\sqrt{5}-5)/2,\sqrt{5}+2,(5\sqrt{5}
+9)/2)$;
in all cases $\Gamma$ is isomorphic to $H[2;2,3;5]$.
\item
$(\beta,\gamma,\beta')=((\sqrt{5}-5)/2,(\sqrt{5}-1)/2,(3\sqrt{5}
-1)/2)$;
$\Gamma$ is isomorphic to $\Tet[3,3;5]$.
\item
$\beta=-3$, $\gamma=2\cos(2\pi/m)$, where $m\geq 5$, $(m,4)=1$,
and
$\beta'=2\gamma$;
$\Gamma$ is isomorphic to $\Tet[4,m;3]$.
\item
$n\geq 5$, $(n,3)=1$, $\gamma=2(\beta+3)$ and
$\beta'=-6(2\cos(\pi/n)+\beta+2)/\beta$;
$\Gamma$ is isomorphic to $R[n,2;2]$.
\item
$\beta=-3$, $\gamma=2\cos(2\pi/m)$, where $m\geq 8$, $(m,4)=2$,
and
$\beta'=2\gamma$; $\Gamma$ is isomorphic to $H[m;3,3;2]$.
\item
$\beta=-3$, $\gamma=2\cos(2\pi/m)-1$, where $m\geq 4$, $(m,3)=1$,
and
$\beta'=\gamma^2+4\gamma$;
$\Gamma$ is isomorphic to $\Tet[2,3,3;2,3,m]$.
\end{enumerate}
\end{theorem}

\begin{proof}
Now we shall use fundamental polyhedra for $\Gamma^*$
described in \fullref{construction_intersect} for $n$ odd and
the Poincar\'e
theorem to find a presentation for $\Gamma$.

{\bf 1.} \qua
Let $\Gamma$ have parameters as in row $P_4$ of
\fullref{table_param}.
Consider the polyhedron bounded by
$\alpha$, $\alpha'$, $R_\omega(\alpha)$, $R_\omega(\alpha')$,
$\zeta$,
$R_\omega(\zeta)$, $e_g(\zeta)$ and $R_\omega(e_g(\zeta))$,
which is the union of four copies of $\P$ shown in \fullref{fp_odd}(a).
Its faces are identified by~$f$, $f'=e_gf^{-1}e_g$ and two loxodromic
elements
$L=e_ge_1$ and $L'=e_ge_1f^{-1}=Lf^{-1}$.
Using the Poincar\'e polyhedron theorem, one can show that
$$
\langle f,f',L,L'\rangle=
\langle
f,L|f^n,(LfL^{-1}f)^{m/2},(fL^{-1}fL^{-1}fL^2)^\ell\rangle.
$$
Obviously, $\langle f,L\rangle\cong \mathcal{S}_3[n,m/2,\ell]$.
Further, since $L=e_ge_1=gef^{(n-1)/2}e=gf^{-(n-1)/2}$,
the group $\langle \! f,L \! \rangle$ coincides with $\Gamma$.
Hence, $\Gamma$ is
isomorphic to
$\! \mathcal{S}_3[n,\!m/2,\! \ell]$, where $m$ is even ($1/n+1/m\leq 1/2$),
$m=\infty$
or $m=\overline\infty$ and $\ell\geq 2$ is an integer.
The orbifold $Q(\Gamma)$ is shown in \fullref{s2s1}(c).

{\bf 2.} \qua
Let  $\Gamma$ have parameters as in row $P_5$.
Denote $y=R_\delta R_\omega$, $\zeta_1'=R_\omega(\zeta_1)$
and consider the polyhedron
$\P(\alpha,R_\omega(\alpha),\xi,\zeta_1,\zeta_1',y(\zeta_1),y(\zeta_1'))$,
which is the union of four copies of the polyhedron $\Q$ shown in
\fullref{fp_odd}(b).
Its faces are identified by the transformations $f$, $v=ye_g$,
$u=ye_1$ and $u'=yfe_1$. As usual, we apply the Poincar\'e polyhedron theorem
to get
$$
\langle f,v,u,u'\rangle=
\langle f,v,u| f^n,v^2,(fv)^m,(fvuvu^{-1})^k,[f,u]\rangle.
$$
We see that $\langle f,v,u\rangle\cong GTet_2[n,m,k]$, where
the isomorphism is given by $f\mapsto x^{-1}$, $v\mapsto y$,
$u\mapsto z^{-1}$. 
So it remains to show
that $\langle f,v,u\rangle$ is actually generated by $f$ and $g$.

First, note that since the axis of $y$ 
is orthogonal to the axis of $f$,
$yfy=f^{-1}$. Now since $m$ is odd and $v^2=1$, we can write
\begin{align*}
1&=(fv)^m=(fvfv)^{(m-1)/2}fv=(fe_gyfye_g)^{(m-1)/2}fv\\
&=(fe_gf^{-1}e_g)^{(m-1)/2}fv=(fgfg^{-1})^{(m-1)/2}fv.
\end{align*}
Therefore, $v=(fgfg^{-1})^{(m-1)/2}f\in\Gamma$.
Further, $u=ye_1=vgef^{(n-1)/2}e=vgf^{-(n-1)/2}$ 
and hence $u\in\Gamma$. So we have shown that
$\langle f,v,u\rangle\subseteq\Gamma$. On the other hand,
$g=vuf^{(n-1)/2}$ and hence $\Gamma\subseteq\langle f,v,u\rangle$.
Then $\Gamma=\langle f,v,u\rangle\cong GTet_2[n,m,k]$, 
where $m$ is odd
($1/n+1/m<1/2$)
and $k\geq 2$ is an integer, $\infty$ or $\overline\infty$.

Now suppose that $\Gamma$ has parameters as in row $P_6$.
This case is similar to the case of parameters $P_5$, but
technically it is more complicated.

Denote $t=R_\delta R_\tau$, $y=R_\alpha R_\delta$ and
$v=yt$. Consider the polyhedron bounded by
$\omega$, $R_\alpha(\omega)$, $\tau$, $\zeta_2$, $R_\alpha(\zeta_2)$,
$y(\zeta_2)$ and $y(R_\alpha(\zeta_2))$, 
which is the union of four copies of the polyhedron
shown in \fullref{p1_hyp}(a). Its faces are identified by 
$f$, $v$, $u=ye_1$ and $u''=ye_1f$. 
Again, by the Poincar\'e polyhedron theorem, we get
$$
\langle f,v,u,u''\rangle=
\langle f,v,u| f^n,v^2,(fv)^3,(fvuvu^{-1})^k,[f,u]\rangle\cong GTet_2[n,3,k].
$$
Further,
$$
u=ye_1=(ye_g)(e_ge_1)=
(R_\alpha R_\delta)(R_\delta R_\xi)(gef^{(n-1)/2}e)=h_1^{-1}gf^{-(n-1)/2},
$$
where $h_1=R_\xi R_\alpha$.
Note that $h_1=vf^{-2}v$.
Then
$$
g=h_1uf^{(n-1)/2}=vf^{-2}vuf^{(n-1)/2}.
$$
Hence,
$\Gamma$ is a subgroup of $\langle f,v,u\rangle$. Now
one can apply the Todd--Coxeter algorithm, 
see eg Johnson \cite{Joh90}
to $\langle f,v,u\rangle$ and
its subgroup $\Gamma$ generated by $f$ and $g=vf^{-2}vuf^{(n-1)/2}$
to show that $|\langle f,v,u\rangle:\Gamma|=1$, ie,
$\langle f,v,u\rangle=\Gamma$.

The orbifold with the fundamental group $GTet_2[n,m,k]$
is shown in \fullref{s2s1}(b).

{\bf 3.} \qua
If $\Gamma$ has parameters as in row $P_9$, we consider
the polyhedron bounded by $\omega$, $\xi$, $R_\alpha(\omega)$,
$R_\alpha(\xi)$, $\mu$ and $R_\delta(\mu)$
(compare with \fullref{p1_hyp}(b))
whose faces are identified by $f$, $h_1=R_\xi R_\alpha$ and
$z=R_\delta R_\mu$.
Then $\langle f,h_1,z\rangle$ has the presentation
\begin{equation}\label{pres}
\langle f,h_1,z|f^3,h_1^m,(f^{-1}h_1)^2,[f,z]^2,[h_1,z]\rangle
\end{equation}
Hence, $\langle f,h_1,z\rangle\cong GTet_1[m,3,2]$, where the isomorphism
is given by $f\mapsto xy$, $h_1\mapsto x$, $z\mapsto z$.
Let us show that $\langle f,h_1,z\rangle=\Gamma$.

Denote $a=R_\mu R_\alpha$, $b=R_\delta R_\alpha$ and
$s=R_\alpha R_{\zeta_1}$. Then $z=ba$.
Since the axis of $b$ is orthogonal to $f$, we have
$bfb=f^{-1}$ and, since $\mu$ is orthogonal to $\alpha$, we have
that $a^2=1$.
From the decomposition of the link made by
$\alpha$, $\omega$ and $\zeta_1$ by the reflection planes, we obtain
$$
e_1=sas^{-1} \text{ and } s=afa.
$$
Therefore, $e=f^{-1}e_1=f^{-1}sas^{-1}=f^{-1}afaf^{-1}a$ and
$$
g=e_ge=h_1bf^{-1}afaf^{-1}a=h_1fzfz^{-1}fz.
$$
So we have shown that $\Gamma$ is a subgroup of 
$\langle f,h_1,z\rangle$.
Now it is sufficient to apply the Todd--Coxeter algorithm 
to the group $\langle f,h_1,z\rangle$ given by presentation
\eqref{pres} and its subgroup
$\langle f,g\rangle$
to see that
$\langle f,h_1,z\rangle=\Gamma$.

Thus $\Gamma\cong GTet_1[m,3,2]$
and the orbifold $Q(\Gamma)$ is shown in \fullref{s2s1}(d).

{\bf 4.} \qua
Consider the groups with parameters as in rows
$P_{11}$--$P_{15}$, $P_{17}$, $P_{18}$. In all of these cases $R_\zeta\in\Gamma^*$.
We know a fundamental polyhedron for $\Gamma^*$ and the structure of $\Gamma^*$
in each case. Since all these polyhedra are obtained as decompositions
of $\P$ into smaller polyhedra, they have common properties. Namely,
\begin{itemize}
\item[(P1)]
the elements $f'=R_{\alpha'}R_\omega$ and $h=R_{\alpha'}R_\alpha$
belong to $\Gamma$. Indeed,
\begin{align*}
f'& =R_{\alpha'}R_\omega=e_g R_\alpha e_g R_\omega  =(e_g R_\alpha R_\omega)
(R_\omega e_g R_\omega)\\
& =e_gf^{-1}e_g=gfg^{-1}
\end{align*}
and $h=f'f=gfg^{-1}f$.
\item[(P2)]
the elements $h_2=R_{\alpha'}R_\zeta$, $t_1=R_\alpha R_\zeta$ and
$t_2=R_\omega R_\zeta$ belong to $\Gamma$.

Denote $\alpha''=e_1(\alpha')$.
Note that $R_{\alpha''}=e_1 R_{\alpha'}e_1$
and $R_\omega=e_1 R_\alpha e_1$. Then
\begin{align*}
h_2^2&=R_{\alpha'}R_{\alpha''}=(R_{\alpha'}R_\alpha)(R_\alpha R_{\alpha''})=
hR_\alpha e_1 R_{\alpha'}e_1=he_1R_\omega R_{\alpha'} e_1\\
&=he_1gf^{-1}g^{-1}e_1
=gfg^{-1}f^{-(n-1)/2}g^{-1}fgf^{-(n-1)/2}.
\end{align*}
Since $h_2$ always has odd order for the groups with parameters
$P_{11}$--$P_{15}$, $P_{17}$, $P_{18}$, the fact that $h_2^2\in\Gamma$ implies 
$h_2\in\Gamma$.
Further, since $t_1=(R_\alpha R_{\alpha'})(R_{\alpha'} R_\zeta)=h^{-1}h_2$
and 
$t_2=R_\omega R_\zeta=e_1 R_\alpha e_1 R_\zeta=e_1 R_\alpha R_\zeta e_1=e_1t_1e_1$,
both $t_1$ and $t_2$ belong to~$\Gamma$.
\end{itemize}

For the groups with parameters $P_{12}$, $P_{13}$, $P_{15}$, $P_{17}$ or $P_{18}$,
$\Gamma^*=\langle G_T,e_g\rangle$, where $e_g$ coincides with the axis of a
$\mathbb Z_2$--symmetry of $T$ (see Figures \ref{gamma9-16}(b), \ref{gamma_central}(a),
\ref{gamma_central}(c)--(e)).
Then $\widetilde\Gamma=\langle \Delta_T,e_g\rangle$, where $\Delta_T$ is the
orientation preserving subgroup of~$G_T$.
Proceeding as in the proof of the property (P2), one can show that
the rotations from $\Delta_T$ belong to $\Gamma$.
In particular, since $e_1$ passes through an edge of $T$,
$e_1\in\Gamma$ and, therefore, $e\in\Gamma$.
Thus, $\Gamma=\widetilde\Gamma$. If $T$ is a compact tetrahedron,
it was shown by Derevnin and Mednykh \cite{DM98} that each
$\langle \Delta_T,e_g\rangle$ is isomorphic to some $H[p;n,m;q]$.
It is easy to see that the same is true for non-compact~$T$.
It remains to find $p$, $n$, $m$ and $q$, which is not difficult to
do since the position of $e_g$ is known in each case.
For example, if $\Gamma$ has parameters $P_{12}$,
$\Gamma\cong H[2;2,3;5]$.

Now consider the groups with parameters as in row $P_{14}$
(see \fullref{gamma_central}(b)).
In this case
$T=T[2,2,4;2,3,5]$.
Denote by $\rho$ the reflection plane through $e_g$ and $\alpha'\cap\zeta$
and let $\overline h_2=R_\rho R_\zeta$. Then  $\overline h_2^2=h_2$.

It is clear that $\Delta_T$ is generated by $e_1$, $t_1$ and $\overline h_2$
and has the presentation
$$
\langle  e_1, t_1, \overline h_2| e_1^2,t_1^3, \overline h_2^5,
(e_1 t_1)^4, (e_1 \overline h_2)^2, (t_1^{-1}\overline h_2)^2\rangle
\cong \Tet[4,5;3].
$$
Let us show that $\Delta_T=\Gamma$. From the link of the vertex made by
$\alpha$, $\omega$ and $\zeta$, we see that
$f=t_2t_1^{-1}=e_1t_1e_1t_1^{-1}$. Since $e_g=\overline h_2t_1$,
$g=e_ge=\overline h_2 t_1 f^{-1} e_1=\overline h_2t_1^2e_1 t_1^{-1}$.
Therefore, $\langle f,g\rangle$ is a subgroup of $\Delta_T$. Furthermore,
since $\overline h_2$ is of odd order, the Todd--Coxeter method gives us that
$\langle f,g\rangle$ coincides with~$\Delta_T$. Thus $\Gamma\cong \Tet[4,5;3]$.

We remark that the case of the groups with parameters as in row $P_{11}$ with
$(r,4)=1$ is analogous to the case of the groups with parameters as in row
$P_{14}$ with the difference that $h$ is hyperbolic and
$T=T[2,2,4;2,3,r]$
is an infinite volume tetrahedron. The group $\Gamma$ is then
isomorphic to $\Tet[4,r;3]$, where $r\geq 7$ is odd.

Consider the groups with parameters $P_{11}$ with $(r,4)=2$.
The consideration is quite delicate so we shall do it in detail.

Let $\kappa$ be the reflection plane passing through $e_1$ orthogonally
to $\zeta$ (see \fullref{gamma9-16}(a)), let $\tau$ be the plane through
$e_g$ and $t_1$, and let $\rho$ again be the plane through $e_g$
and $\alpha'\cup\zeta$.
Denote $s=R_\rho R_\kappa$, $s'=R_\tau s R_\tau$,
$u=R_\tau R_{\alpha'}$ and consider
$\P(\alpha,\zeta,\kappa,\alpha',R_\tau(\alpha'),R_\tau(\kappa))$.
Its faces are identified by $s$, $s'$, $t_1$ and $u$. Then by the Poincar\'e
polyhedron theorem we get the presentation
$$
G=\langle s,s',t_1,u| s^2,(s')^2,t_1^3,u^3,(t_1u)^{r/2},
ust_1s',(ss')^2\rangle.
$$
Since $s'=ust_1$,
$$
G=\langle s,t_1,u| s^2,t_1^3,u^3,(t_1u)^{r/2},
(sust_1)^2,(ust_1)^2\rangle\cong H[r/2;3,3;2],
$$
where $r/2\geq 5$ and $r/2$ is odd.
We claim that $G=\Gamma$.

Note that $R_\omega=R_\kappa R_\tau R_\kappa$ and
$R_\alpha=R_\tau R_\zeta R_\tau$. Therefore,
\begin{align*}
f&=R_\omega R_\alpha=R_\kappa R_\tau R_\kappa R_\tau R_\zeta R_\tau=
((R_\kappa R_\rho)(R_\rho R_\tau))^2(R_\zeta R_\tau)\\
&=(se_g)^2t_1=ss't_1,
\end{align*}
because $e_g s e_g=s'$. 
Hence, $f=sust_1^2$.
Denote as before $\overline h_2=R_\rho R_\zeta$ and 
$h_2=\overline h_2^2$.
Since $e_1=\overline h_2^{-1}s$, we get
$e=e_1f=\overline h_2^{-1}ust_1^2$.
Therefore, since $\overline h_2^2=h_2=u^{-1}t_1^{-1}$
and $e_g=t_1^{-1}\overline h_2^{-1}$, we obtain that
$$
g=e_ge=t_1^{-1}\overline h_2^{-2}ust_1^2=u^2st_1^2.
$$
So we have proved that $\Gamma$ is a subgroup of $G$.

On the other hand, since $h_2\in\Gamma$ and $t_1\in\Gamma$
(see the property (P2)),
we get that $u=t_1^{-1}h_2^{-1}\in\Gamma$ and, therefore,
$s=u^{-2}gt_1^{-2}=ugt_1\in\Gamma$.
Thus, $\Gamma=G$.

{\bf 5.} \qua
Consider the remaining cases of groups, with parameters as in rows
$P_7$, $P_8$, $P_{10}$, $P_{16}$ and $P_{19}$.
In all of these cases $R_\zeta\notin\Gamma^*$.
As in part 4 of the proof, the elements $h=R_{\alpha'}R_\alpha$ and
$h_3=R_{\alpha'}R_{\alpha''}$ belong to $\Gamma$.
Denote $\sigma=e_1(\eta)$, $a=R_\eta R_\alpha$ and $b=R_\alpha R_\sigma$.

Suppose $\Gamma$ has parameters as in row $P_7$.
Consider the tetrahedron $T=T[2,3,n;2,3,n]$ bounded by
$\alpha$, $\omega$, $\eta$ and $\sigma$.
Denote $s=R_\sigma e_g R_\sigma$ and $t=e_1 s$.
Then $t$ passes through the ``midpoints'' of the edges with
dihedral angles of $\pi/2$ and all $e_1$, $s$ and $t$ are the axes of
${\mathbb Z}_2$--symmetries of $T$.

It is clear that the group $\Delta_T$, which is the orientation
preserving subgroup of $G_T$, is generated by $f$, $a$ and $b$.
Let $H=\langle \Delta_T,t\rangle$. We leave a proof of the fact that
$H\cong H[2;3,n;2]$ as an exercise for the reader, but
we prove that $H=\Gamma$.
Let $\sigma'=e_g(\sigma)$. Then $R_{\sigma'}=e_g R_\sigma e_g$
and $R_\sigma R_{\sigma'}=R_\omega R_\sigma$.
Therefore,
$$
e_1 e_g=(e_1 s)(se_g)=tsR_\sigma R_\sigma e_g=tR_\sigma e_g R_\sigma e_g
=tR_\sigma R_{\sigma'}=tR_\omega R_\sigma=tfb.
$$
Hence, $g=e_ge=e_ge_1f^{(n-1)/2}=(tfb)^{-1}f^{(n-1)/2}\in H$.
So $\Gamma\subseteq H$.

On the other hand, $e_1e_g=f^{(n-1)/2}g^{-1}$.
Denote $\overline h_3=R_\sigma R_\eta=R_{\alpha'}R_\sigma$.
We know that
$\overline h_3^3=h_3\in \Gamma$. Since $(n,3)=1$, $\overline h_3\in\Gamma$.
Then $b=(R_\alpha R_{\alpha'})(R_{\alpha'} R_\sigma)=h^{-1}\overline h_3\in\Gamma$,
$a=\overline h_3^{-1} b^{-1}\in\Gamma$ and $t=f^{(n-1)/2}g^{-1}b^{-1}f^{-1}\in\Gamma$.
Thus, $H=\langle f,a,b,t\rangle$ is a subgroup of $\Gamma$. So
$\Gamma=H\cong H[2;3,n;2]$.

Suppose $\Gamma$ has parameters as in row $P_8$.
Let now $\kappa$ be the reflection plane such that $e_g=R_\kappa R_\sigma$
and let $s=R_\kappa e_1 R_\kappa$.
Let $\rho$ be the plane through $M$, $K$ and $L$, and 
let $\tau=R_\kappa(\rho)$
(see \fullref{gamma_noncentral}(b)).
Then $\tau$ passes through $M$, $K$ and $L'$,
where $L'=R_\kappa(L)$, and $s$ lies in $\rho$.
Moreover, the sum of the angles that $\rho$ makes 
with $\alpha$ and $\sigma$
equals $\pi$ and $\rho$ intersects $\tau$ orthogonally.
Consider the polyhedron $\P=\P(\alpha,\eta,\alpha',R_\sigma(\alpha),\rho,\tau,
R_\sigma(\rho),R_\sigma(\tau))$. 
Its faces are identified by
$\overline h_3$, $z=e_g \overline h_3e_g$, $u=se_g$ and $v=e_1e_g$.
Using the Poincar\'e polyhedron theorem, we get the presentation
$$
H=\langle z,\overline h_3, u,v|z^n,\overline h_3^n,
(\overline h_3 z)^2,vu\overline h_3,uvz^{-1},(uv^{-1})^2\rangle.
$$
Since $z=uv$ and $\overline h_3=u^{-1}v^{-1}$,
$$
H=\langle u,v|(uv)^n,(uv^{-1})^2,[u,v]^2\rangle\cong R[n,2;2].
$$
Let us show that $H=\Gamma$. First, note that since $(n,3)=1$,
$\overline h_3\in\Gamma$. Therefore, since
$v=e_1e_g=f^{(n-1)/2}g^{-1}\in \Gamma$ and $u=v^{-1}\overline h_3\in\Gamma$,
$H\subseteq\Gamma$.

In order to express $f$ and $g$ in terms of $u$ and $v$, we recall that
$R_\sigma=e_1 R_\eta e_1$ and $R_\omega=e_1 R_\alpha e_1$ and note that
$z=R_\kappa \overline h_3^{-1} R_\kappa=R_\alpha R_\sigma$.
Then
$$
s\overline h_3s=R_\kappa e_1R_\kappa\overline h_3 R_\kappa e_1 R_\kappa
=R_\kappa e_1z^{-1}e_1 R_\kappa=R_\kappa e_1 R_\sigma R_\alpha e_1
R_\kappa=R_\kappa R_\eta R_\omega R_\kappa.
$$
Furthermore, $R_\alpha=R_\kappa R_\eta R_\kappa$ and, since $\kappa$
is orthogonal to $\omega$, $R_\omega R_\kappa=R_\kappa R_\omega$.
Hence,
$s\overline h_3s=R_\alpha R_\kappa R_\omega R_\kappa=R_\alpha R_\omega
=f^{-1}$.
On the other hand, $s\overline h_3s=se_gze_gs=u^2vu^{-1}$.
Thus, $f=uv^{-1}u^{-2}$ and $g=v^{-1}f^{(n-1)/2}=v^{-1}(uv^{-1}u^{-2})^{(n-1)/2}$.
So we have shown that $\Gamma\subseteq H$ and, hence, $\Gamma=H$.

Gluing the faces of $\P$ by $\overline h_3$, $z$, $u$ and $v$, we obtain
the orbifold embedded in ${\mathbb R}P^3$  (see \fullref{rp3}).

\begin{figure}[ht!]
\centering
\labellist
\small
\pinlabel {$n$} at 49 49
\pinlabel {$2$} at 24 40
\pinlabel {$n$} at 48 10
\pinlabel {\turnbox{20}{$\pi/2$}} at 71 24
\endlabellist
\includegraphics[width=5 cm]{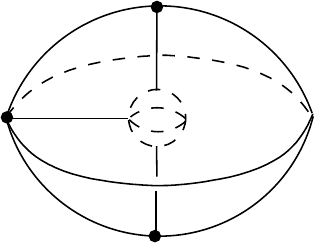}\\
$\pi_1^{\orb}(Q)\cong R[n,2;2]$
\caption{Underlying space is $\mathbb{R}P^3\backslash{\mathbb
B}^3$}\label{rp3}
\end{figure}

Suppose that $\Gamma$ has parameters as in row $P_{10}$.
Consider the tetrahedron $T=T[2,3,3;2,3,m/2]$ bounded by $\alpha$,
$\omega$, $\eta$ and~$\sigma$. The group $\Delta_T$ has the presentation
$$
\Delta_T=\langle f,a,b| f^3,a^2,b^3,(af)^3,(bf)^2,(ab)^{m/2}\rangle
\cong \Tet[2,3,3;2,3,m/2].
$$
We shall show that $\Delta_T=\Gamma$. Note that $e_g=b^{-1}e_1b$.
Then
$$
e_1be_1=e_1R_\alpha R_\sigma e_1=e_1 R_\alpha e_1 R_\eta=R_\omega R_\eta
=fa.
$$
Therefore,
$$
g=e_ge_1f=b^{-1}e_1be_1f=b^{-1}faf.
$$
Hence, $\Gamma=\langle f,g \rangle$ is a subgroup of $\Delta_T$. Applying
the Todd--Coxeter algorithm, we see that, since $(m/2,3)=1$,
$\langle f,g\rangle=\Delta_T$.

Similarly, one can show that in the case of the parameters of type
$P_{16}$, $\Gamma=\Delta_T$, where $T=T[2,2,3;2,5,3]$ and
$\Delta_T\cong \Tet[3,3;5]$.

Suppose that $\Gamma$ has parameters as in row $P_{19}$.
In this case $\Gamma^*=\langle G_T,e_g\rangle$, where $T=T[2,2,3;2,5,3]$.
Then $\widetilde \Gamma=\langle \Delta_T,e_g\rangle$.
Notice that all rotations from $\Delta_T$ belong to $\Gamma$, in
particular, $e_1\in\Gamma$. Hence, $e\in\Gamma$ and, therefore,
$\Gamma=\widetilde\Gamma$. 
It was shown by Derevnin and Mednykh \cite{DM98} that
$\langle \Delta_T, e_g\rangle\cong H[2;2,3;5]$. Thus,
$\Gamma\cong H[2;2,3;5]$.
\end{proof}

\begin{rem}
When $Q=Q(\mathcal{S}_3[n,m,q])$, the singular set of $Q$ is placed
into
$\mathcal{S}(3)$ in such a way that the curve consisting of cone
points
of indices $n$ and $m$ lies on a regular fiber.
When $Q=Q(GTet_2[n,m,q])$, the curve consisting of cone points of
indices
$m$ and $2$ lies on a regular fiber and the singular component
of
index $n$ lies on a critical fiber.

In \fullref{rp3}, $\mathbb{R}P^3$ is shown as a lens with
antipodal points on the boundary identified. The angle at 
the edge of the lens is $\pi/2$ and, therefore, the edge is mapped
onto a singular loop with index 2.
\end{rem}
\bibliographystyle{gtart}
\bibliography{link}

\end{document}